\documentclass[a4paper,reqno,11pt]{amsart}

\usepackage{amsmath, amsfonts, amssymb, amsthm, amscd}
\usepackage{graphicx}
\usepackage{psfrag}
\usepackage{perpage}
\usepackage{url}
\usepackage{color}

\usepackage{natbib}

\usepackage[utf8]{inputenc}

\usepackage[a4paper,scale={0.72,0.74},marginratio={1:1},footskip=7mm,headsep=10mm]{geometry}

\usepackage{hyperref}

\numberwithin{equation}{section}

\newtheorem{theorem}{Theorem}[section]
\newtheorem{lemma}[theorem]{Lemma}
\newtheorem{proposition}[theorem]{Proposition}
\newtheorem{corollary}[theorem]{Corollary}
\newtheorem{remark}[theorem]{Remark}
\newtheorem{definition}[theorem]{Definition}

\newtheorem{assumption}[theorem]{Assumption}

\newcommand{\capT}{{\bar{T}}}

\renewcommand{\tilde}{\widetilde}          
\DeclareMathSymbol{\leqslant}{\mathalpha}{AMSa}{"36} 
\DeclareMathSymbol{\geqslant}{\mathalpha}{AMSa}{"3E} 
\DeclareMathSymbol{\eset}{\mathalpha}{AMSb}{"3F}     
\newcommand{\dd}{\text{\rm d}}             

\newcommand{\R}{\mathbb{R}}

\newcommand{\N}{\mathbb{N}}

\renewcommand{\P}{\ensuremath{\mathbb P}}
\newcommand{\E}{\ensuremath{\mathbb E}}

\newcommand{\ind}{{\sf 1}}

\renewcommand{\epsilon}{\varepsilon} 
\renewcommand{\theta}{\vartheta} 
\renewcommand{\rho}{\varrho} 

\newenvironment{myenumerate}{%
\renewcommand{\theenumi}{\arabic{enumi}}%
\renewcommand{\labelenumi}{{\rm(\theenumi)}}%
\begin{list}{\labelenumi}
	{%
	\setlength{\itemsep}{0.4em}%
	\setlength{\topsep}{0.5em}%
	\setlength\leftmargin{2.45em}%
	\setlength\labelwidth{2.05em}%
	\setlength{\labelsep}{0.4em}%
	\usecounter{enumi}%
	}%
	}%
{\end{list}
}

{\end{list}
}

{\end{myenumerate}}

\newenvironment{myitemize}{%
\begin{list}{$\bullet$}%
 	{%
	\setlength{\itemsep}{0.4em}%
	\setlength{\topsep}{0.5em}%
	\setlength\leftmargin{2.45em}%
	\setlength\labelwidth{2.05em}%
	\setlength{\labelsep}{0.4em}%
	}%
	}%
{\end{list}}

\renewenvironment{itemize}{
\begin{myitemize}}%
{\end{myitemize}}

\MakePerPage[2]{footnote} 

\title[Wasserstein distance between the marginals of two Markov processes]{Evolution of the Wasserstein distance between the marginals of two Markov processes}
\author{Aur\'elien Alfonsi, Jacopo Corbetta and Benjamin Jourdain}
\thanks{Universit\'e Paris-Est, Cermics (ENPC), INRIA, F-77455 Marne-la-Vall\'ee, France, e-mails : aurelien.alfonsi@enpc.fr, jacopo.corbetta@enpc.fr, benjamin.jourdain@enpc.fr. This research benefited
     from the support of the ``Chaire Risques Financiers'', Fondation du
     Risque, the French National Research Agency under the program
  ANR-12-BS01-0019 (STAB)}
\date{\today}

\newcommand{\opsi}{\widetilde \psi}
\newcommand{\oX}{\widetilde X}

\newcommand{\oL}{\widetilde{L}}
\newcommand{\olambda}{\widetilde{\lambda}}

\newcommand{\ok}{\widetilde{k}}
\newcommand{\oP}{\widetilde{P}}

\begin{document}
\begin{abstract}
  In this paper, we are interested in the time derivative of the Wasserstein distance between the marginals of two Markov processes. As recalled in the introduction, the Kantorovich duality leads to a natural candidate for this derivative. Up to the sign, it is the sum of the integrals with respect to each of the two marginals of the corresponding generator applied to the corresponding Kantorovich potential. For pure jump processes with bounded intensity of jumps, we prove that the evolution of the Wasserstein distance is actually given by this candidate. In dimension one, we show  that this remains true for Piecewise Deterministic Markov Processes. We apply the formula to estimate the exponential decrease rate of the Wasserstein distance between the marginals of two birth and death processes with the same generator in terms of the Wasserstein curvature.
 \end{abstract}
 \maketitle


\section{Introduction}

The goal of this paper is to compute the time derivative of the $\rho$-Wasserstein distance
between the marginals of two Markov processes $\{X_t\}_{t\ge 0}$ and $\{\oX_t\}_{t\ge 0}$. For
$\rho\geq 1$, the $\rho$-Wasserstein distance between two probability measures $P, \tilde{P}$ on $\R^d$ is defined as
\begin{equation}
 W_{\rho}(P,\tilde{P})=\left(\inf_{\pi\in \Pi(P,\tilde{P})}\int_{\R^d\times\R^d}|x-y|^{\rho}\pi(\dd x,\dd y)\right)^{1/\rho}\,
\end{equation}
where $\Pi(P,\tilde{P})$ is the set of probability measures on $\R^d\times\R^d$ with respective marginals $P$ and $\tilde{P}$. It is well known that there exists $\pi\in\Pi(P,\tilde{P})$ such that $W^\rho_{\rho}(P,\tilde{P})=\int_{\R^d\times\R^d}|x-y|^{\rho}\pi(\dd x,\dd y)$ (see for instance Theorem 3.3.11 of \cite{cf:RR}).


It is possible to prove (see for instance Theorem 5.10 of \cite{cf:Villani}) that: 
\begin{equation}
W_{\rho}^\rho(P,\tilde{P})= \sup\left\{-\int_{\R^d}\phi(x)P(\dd x)-\int_{\R^d}\tilde{\phi}(y)\tilde{P}(\dd y) \right\}
\end{equation}
where the supremum runs along all pairs $(\phi,\tilde{\phi})\in L^1(P)\times L^1(\oP)$ such that $\forall (x,y)\in\R^d\times\R^d$, $-\phi(x)-\tilde{\phi}(y)\le |x-y|^\rho$. Moreover, according to Theorem 5.10 of \cite{cf:Villani}, if $\int_{\R^d}|x|^\rho P(\dd x)+\int_{\R^d}|y|^\rho\tilde{P}(\dd y)<+\infty$,  there exists a couple $(\psi,\opsi)\in L^1(P)\times L^1(\oP)$
of $\rho$-convex functions 
such that $W_{\rho}^\rho(P,\tilde{P})= -\int_{\R^d}\psi(x)P(\dd x)-\int_{\R^d}\tilde{\psi}(y)\tilde{P}(\dd y) $ and one is the $\rho$-transform of the other, i.e.
\begin{equation}
   \psi(x)=-\inf_{y\in\R^d}\left\{
 |x-y|^\rho+\opsi(y) \right\} \quad \text{and}\quad 
 \opsi(y)=-\inf_{x\in\R^d}\left\{
 |x-y|^\rho+\psi(x) \right\}\,.\label{dual}
\end{equation}
These functions
$\psi$ and $\opsi$ are called Kantorovich potentials. For an optimal coupling $\pi$ (which is unique when $\rho>1$ according to Theorem 6.2.4 of \cite{cf:AGS}), since $\int_{\R^d\times\R^d}-(\psi(x)+\opsi(y))\pi(\dd x,\dd y)=W^\rho_{\rho}(P,\tilde{P})=\int_{\R^d\times\R^d}|x-y|^{\rho}\pi(\dd x,\dd y)$ and $-(\psi(x)+\opsi(y))\le |x-y|^\rho$, we necessarily  have
\begin{equation}
   \pi(\dd x,\dd y)\mbox{ a.e.}, -\psi(x)-\opsi(y)= |x-y|^\rho.\label{optipi}
\end{equation}
In dimension $d=1$, an optimal coupling~$\pi$ is the probability law of $(F^{-1}(U),\tilde{F}^{-1}(U))$ where $F(x)=P((-\infty,x])$ and $\tilde{F}=\oP((-\infty,x])$ are respectively the cumulative distribution functions of $P$ and $\tilde{P}$,   $F^{-1}(u)=\inf \{ x \in \R, F(x)>u \}$ and $\tilde{F}^{-1}(u)=\inf \{ x \in \R, \tilde{F}(x)>u \}$ their right-continuous pseudo-inverse,  and $U$ is a uniform random variable on $[0,1]$ (see Theorem 3.1.2 of \cite{cf:RR}). When $F$ is continuous and $\rho>1$, we explicit a pair of Kantorovich potentials. Let us define \begin{equation}\label{potential_1d}\psi(x)=\rho\int_0^x|T(x')-x'|^{\rho-2}(T(x')-x')\dd x', \ T(x)=\tilde{F}^{-1}(F(x)).
\end{equation}
and $\tilde{\psi}(y)=-\inf \{ |x-y|^\rho+\psi(x)\}$. 
 Since $r \mapsto \rho|r|^{\rho-2}r$ and $x\mapsto T(x)$ are  nondecreasing, we have for $x,z\in \R$,
\begin{align*}
 \psi(x)-\psi(z)&=\rho\int_z^x|T(x')-x'|^{\rho-2}(T(x')-x')\dd x' \\
&\le \rho\int_z^x|T(x)-x'|^{\rho-2}(T(x)-x')\dd x'=|T(x)-z|^\rho-|T(x)-x|^\rho.
\end{align*}
Then, we have \begin{equation}\label{optim1d}
   \psi(x)+|T(x)-x|^\rho=\inf_{z\in \R} \{\psi(z)+|T(x)-z|^\rho \}=-\tilde{\psi}(T(x)). 
\end{equation}
Thus, $\pi(\dd x,\dd y)=P(\dd x) \delta_{T(x)}(\dd y)$ is the optimal coupling and $T$ is called the optimal transport map.

Let us now consider two $\R^d$-valued Markov processes $\{X_t\}_{t\geq 0}$ and $\{\oX_t\}_{t\geq 0}$ with respective
infinitesimal generators $L$ and $\oL$ such that $\forall t\ge 0$, $\E[|X_t|^\rho+|\oX_t|^\rho]<\infty$. For $t\ge 0$, the $\rho-$Wasserstein distance between the law $P_t$ of $X_t$ and the law $\tilde{P}_t$ of $\oX_t$ is given by
\[
 W_{\rho}^\rho(P_t,\oP_t)=-\int_{\R^d}\psi_t(x) P_t(\dd x) -\int_{\R^d}\opsi_t(y) \oP_t(\dd y)\,,
\]
where $\psi_t,\opsi_t$ are Kantorovich potentials associated with $(P_t,\oP_t)$. 

The above duality gives a very natural candidate for the time derivative of $t\mapsto W_\rho^\rho(P_t,\oP_t)$. Indeed, for $t>0$, let us assume that $(\psi_t,\opsi_t)\in L^1(P_s)\times L^1(\oP_s)$ for $s\in(t-\varepsilon,t+\varepsilon)$ for some $\varepsilon>0$. Then, we have
\begin{equation*}
\begin{split}
 W_\rho^\rho(P_{s},\oP_{s}) &\geq -\int_{\R^d}\psi_t(x)\, P_{s}(\dd x)-\int_{\R^d}\opsi_t(x)\, \oP_{s}(\dd x).
 \end{split}
 \end{equation*}
For $h>0$, choosing $s=t+h$ then (when $t\ge h$) $s=t-h$ leads to
\begin{align*}
 \frac{1}{h}\left(W_\rho^{\rho}(P_{t+h},\oP_{t+h}) - W_\rho^{\rho}(P_t,\oP_t)\right) 
 \geq &-\frac 1h \int_{\R^d}\psi_t(x)(P_{t+h}(\dd x)-P_t(\dd x))\\&-\frac 1h \int_{\R^d}\opsi_t(x)(\oP_{t+h}(\dd x)-\oP_t(\dd x)),\\
\frac{1}{h}\left(W_\rho^{\rho}(P_{t},\oP_{t}) - W_\rho^{\rho}(P_{t-h},\oP_{t-h})\right) 
 \leq &-\frac 1h \int_{\R^d}\psi_t(x)(P_t(\dd x)-P_{t-h}(\dd x))\\&-\frac 1h \int_{\R^d}\opsi_t(x)(\oP_t(\dd x)-\oP_{t-h}(\dd x)),
\end{align*}
If $\psi_t$ and $\opsi_t$ are respectively in the domains of the generators $L$ and $\oL$, 
 one has
\begin{align*}
   \int_{\R^d}\psi_t(x)(P_t(\dd x)-P_s(\dd x))&=\int_s^t \int_{\R^d}L \psi_t(x)P_r(\dd x)\dd r\\
\int_{\R^d}\opsi_t(x)(\oP_t(\dd x)-\oP_s(\dd x))&=\int_s^t \int_{\R^d}\oL \opsi_t(x)\oP_r(\dd x)\dd r.\end{align*}
Plugging these equalities for $s=t+h$ and $s=t-h$ in the previous inequalities and letting $h\to 0$, leads, under continuity (resp. differentiability) at time $t$ of $r\mapsto \int_{\R^d}L \psi_t(x)P_r(\dd x)+\int_{\R^d}\oL \opsi_t(x)\oP_r(\dd x)$ (resp. $r\mapsto W_\rho^\rho(P_r,\oP_r)$), to
\begin{equation}
   \frac{d}{dt}W_\rho^\rho(P_t,\oP_t)= -\int_{\R^d}L \psi_t(x)P_t(\dd x)-\int_{\R^d}\oL \opsi_t(x)\oP_t(\dd x).\label{forderwas}
\end{equation}

In the present paper, we are interested in the slightly weaker integral formula : $\forall 0\le s\le t$,
\begin{equation}
   W_\rho^\rho(P_t,\oP_t)-W_\rho^\rho(P_s,\oP_s)=-\int_s^t\left(\int_{\R^d}L \psi_r(x)P_r(\dd x)+\int_{\R^d}\oL \opsi_r(x)\oP_r(\dd x)\right)\dd r.\label{evowass}
\end{equation}

We will prove that it holds when $\{X_t\}_{t\ge 0}$ and $\{\oX_t\}_{t\ge 0}$ are pure jump Markov processes with bounded intensity of jumps and such that $t\mapsto \E[|X_t|^{\rho(1+\varepsilon)}+|\oX_t|^{\rho(1+\varepsilon)}]$ is locally bounded for some $\varepsilon>0$. The interest of this result is reinforced by the fact that (like in the proof of the Hille-Yoshida theorem)  one can approximate any Markov process using a sequence of 
pure jump Markov processes with  increasing jump intensity. Using this approximation procedure, we check that \eqref{evowass} still holds for one-dimensional Piecewise Deterministic Markov Processes evolving according to an ordinary differential equation between jumps with finite intensity. Even if the case of deterministic processes evolving according to time-dependent ordinary differential equations is covered by~\cite{cf:AGS} (see in particular Theorem 8.4.7), it is not so easy to mix jumps and ODEs. That is why our derivation of (1.6) for PDMPs is restricted to the one-dimensional setting where we can take advantage of the explicit knowledge of the optimal transport map. For diffusion processes, the Euler-Maruyama discretization scheme provides another approximation procedure  considered in~\cite{cf:AJK}.
The interest of an exact formula for the time derivative goes beyond the issue of controlling the Wasserstein stability of a single Markov process, a topic that has been the object of intensive research in the recent years, especially in the diffusion case. Such a formula can also be used to analyze approximation procedures like the Euler-Maruyama discretization scheme studied in~\cite{cf:AJK14,cf:AJK}. Concerning the Wasserstein stability of Markov semi-groups, many authors have studied the Wasserstein contraction property, namely the existence of constants $\kappa,C>0$ such that 
\begin{equation}
   \forall P_0,\tilde{P}_0\mbox{ s.t. }\int_{\R^d}|x|^\rho (P_0+\tilde{P}_0)(\dd x)<\infty,\;\forall t\ge 0,\;W_{\rho}(P_t,\tilde{P}_t)\le Ce^{-\kappa t}W_{\rho}(P_0,\tilde{P}_0).\label{wasstab}
\end{equation}
For the heat semi-group on a smooth Riemannian manifold, \cite{cf:SvR05} proved that the fact the Ricci curvature of the manifold is bounded from below by $\kappa$ is equivalent to \eqref{wasstab} with $C=1$ for all $\rho\ge 1$ or even a single $\rho\ge 1$. Concerning stochastic differential equations with additive Brownian noise 
\begin{equation}
   \dd X_t= \dd W_t+b(X_t)\dd t,\label{edsadd}
\end{equation} 
the Wasserstein contraction with $C=1$ for one or all $\rho\ge 1$ is equivalent to the monotonicity property
$\forall x,y\in\R^d,\;(x-y).(b(x)-b(y))\le -\kappa|x-y|^2$. To check the sufficiency, it is enough to compute $|X_t-\tilde{X}_t|^\rho$ by the It\^o formula where $(X_t)_{t\ge 0}$ and $(\tilde{X}_t)_{t\ge 0}$ are two solutions driven by the same Brownian motion (synchronous coupling) and with respective initial marginals $P_0$ and $\tilde{P}_0$. Recently, considering $W_f(P,\tilde{P})=\inf_{\pi\in \Pi(P,\tilde{P})}\int_{\R^d\times\R^d}f(|x-y|)\pi(\dd x,\dd y)$ for a well chosen increasing concave function $f$ and using the reflection coupling, \cite{cf:E16} was able to prove the exponential decay of $W_f(P_t,\tilde{P}_t)$ and deduce \eqref{wasstab} with $\rho=1$ and $C>1$ for drift functions $b$ satisfying the monotonicity property only outside some ball. In this setting, the restriction of the Wassertein contraction property with $C=1$ and $\rho=2$ to the case when $\tilde{P}_0$ is the invariant probability measure of the SDE had been proved by \cite{cf:BGG12} by some estimation closely related to \eqref{forderwas} (see also \cite{cf:BGG13} for an extension to a class of SDEs nonlinear in the sense of McKean). The $W_1$ contraction by \cite{cf:E16} was extended to \begin{equation}
   \forall \rho\ge 1,\;\forall t\ge 0,\;\forall x,y\in\R^d,\;W_\rho(P_t,\tilde{P}_t)\le Ce^{-\lambda t/\rho}(|x-y|^{1/\rho}1_{\{|x-y|< 1\}}+|x-y|1_{\{|x-y| \ge 1\}})\label{wasscontrmod}
\end{equation} when $(P_0,\tilde{P}_0)=(\delta_x,\delta_y)$ by \cite{cf:LW16a} for SDEs with additive Brownian noise and by \cite{cf:WJ16} when $\dd W_t$ is replaced by the infinitesimal increment of a L\'evy process with L\'evy measure larger than the one of a symmetric stable process. In the case of a pure jump L\'evy noise, \cite{cf:LW16b} proved the Wassertein contraction with $\rho=1$ under some Doeblin condition on the translations of the L\'evy measure local in the translation parameter.
The Wassertein distance $W_f$ with $f(x)=\sqrt{x}$ had already been used by \cite{cf:BCGMZ13} to prove Wasserstein contraction for the piecewise deterministic TCP Markov process. Recently, \cite{cf:WF16} obtained Wasserstein contraction with $C>1$ for a class of diffusion semi-groups generated by weighted Laplacians on Riemannian manifolds with negative curvature and for related stochastic differential equations with multiplicative noise.

The paper is organized as follows. In Section~\ref{sec:kanto} we give 
some results on the integrability properties of the Kantorovich potentials
and in the one-dimensional case ($d=1$) of the translated optimal transport map. In Section~\ref{sec:main}, we state the main result, 
namely formula~\eqref{evowass} for the evolution of the Wasserstein 
distance between the marginals of two pure jump Markov processes. In Section \ref{sec:birth_and_death}, we apply this formula to estimate the Wasserstein distance between the marginals of two birth and death processes with the same generator and obtain \eqref{wasscontrmod} with $\lambda$ equal to the Wasserstein curvature introduced in \cite{cf:J07} (see also \cite{cf:J09} and \cite{cf:CJ13}) to deal with the case $\rho=1$. The proof of the main result relies on integrability properties with respect to the marginal laws of pure jump processes derived in Section~\ref{sec:integra}. 
Finally, in Section~\ref{se:PDMP}, we extend the previous results to one-dimensional Piecewise Deterministic Markov Processes
using an approximation method. 


\section{Integrability properties of the Kantorovich potentials and the translated optimal transport map}\label{sec:kanto}
We first check that when the probability measures $P$ and $\oP$ on $\R^d$ have finite moments of order higher than $\rho$ then these integrability properties are transmitted to any pair $(\psi,\opsi)$ of Kantorovich potentials associated with $W_\rho(P,\oP)$. To get rid of the undetermined additive constant in the definition of  $\psi$ and $\opsi$, we introduce an appropriate generalization of the variance :
\begin{definition}
 For $q\ge 1$, $P$ a probability measure on $\R^d$ and $\varphi\in L^1(P)$, let  $${\mathcal V}^q_P(\varphi)=\int_{\R^d}|\varphi(x)-P(\varphi)|^q P(\dd x)\mbox{ where }P(\varphi)=\int_{\R^d}\varphi(x)P(\dd x).$$
\end{definition}
Notice that $\int_{\R^d}|\varphi(x)|^q P(\dd x)\le 2^{q-1}\left(|P(\varphi)|^q
+{\mathcal V}^q_P(\varphi)\right)$ and ${\mathcal V}^2_P(\varphi)$ is simply the variance of $\varphi$ under the probability measure $P$.
 \begin{proposition}\label{prop:1epsilon}
Let $P,\oP$ be two probability measures on $\R^d$ such that $\int_{\R^d}|x|^{\rho(1+\varepsilon)}P(\dd x) +\int_{\R^d}|y|^{\rho(1+\varepsilon)}\tilde{P}(\dd y)<\infty$ for some $\varepsilon\ge 0$ and let $(\psi,\opsi)$ be a pair of Kantorovich potentials associated with $W_\rho(P,\oP)$. Then $(\psi,\opsi)\in L^{1+\varepsilon}(P)\times L^{1+\varepsilon}(\oP)$ and
$$\max\left({\mathcal V}^{1+\varepsilon}_P(\psi),{\mathcal V}^{1+\varepsilon}_{\oP}(\opsi)\right)\leq 2^{\rho(1+\varepsilon)}\left(\int_{\R^d}|x|^{\rho(1+\varepsilon)}P(\dd x) +\int_{\R^d}|y|^{\rho(1+\varepsilon)}\tilde{P}(\dd y)\right).$$

\end{proposition}

\begin{proof}
Let $\pi\in\Pi(P,\oP)$ be such that $W_\rho^\rho(P,\tilde{P})=\int_{\R^d\times\R^d}|x-y|^{\rho}\pi(\dd x,\dd y)$.
Using \eqref{optipi} for the first and third inequalities and \eqref{dual} for the second and fourth ones, we obtain that $\pi(\dd x,\dd y)\pi(\dd z,\dd w)$ a.e.,
\begin{align*}
   &\psi(x)-\psi(z)\le \psi(x)-\psi(z)-\psi(x)-\opsi(y)=-\psi(z)-\opsi(y)\le |z-y|^\rho\\
&\psi(z)-\psi(x)\le \psi(z)-\psi(x)-\psi(z)-\opsi(w)=-\psi(x)-\opsi(w)\le |x-w|^\rho.
\end{align*}
Therefore we get\begin{equation}\label{ineg_sur_psi}
\pi(\dd x,\dd y)\pi(\dd z,\dd w) \ a.e., \ |\psi(x)-\psi(z)|^{1+\varepsilon}\le \max(|z-y|^{\rho(1+\varepsilon)}, |x-w|^{\rho(1+\varepsilon)}).
\end{equation}
Integrating this inequality with respect to $\pi(\dd x,\dd y)\pi(\dd z,\dd w)$, one obtains
\begin{align*}
   \int_{\R^d\times\R^d}&|\psi(x)-\psi(z)|^{1+\varepsilon}P(\dd x)P(\dd z)\\&\le \int_{\R^d\times\R^d\times\R^d\times\R^d}\left(|z-y|^{\rho(1+\varepsilon)}+|x-w|^{\rho(1+\varepsilon)}\right)\pi(\dd x,\dd y)\pi(\dd z,\dd w)\\
&\le 2^{\rho(1+\varepsilon)-1}\int_{\R^d\times\R^d\times\R^d\times\R^d}\left(|z|^{\rho(1+\varepsilon)}+|y|^{\rho(1+\varepsilon)}+|x|^{\rho(1+\varepsilon)}+|w|^{\rho(1+\varepsilon)}\right)\pi(\dd x,\dd y)\pi(\dd z,\dd w)\\&=2^{\rho(1+\varepsilon)}\left(\int_{\R^d}|x|^{\rho(1+\varepsilon)}P(\dd x) +\int_{\R^d}|y|^{\rho(1+\varepsilon)}\tilde{P}(\dd y)\right).
\end{align*}
The moment ${\mathcal V}^{1+\varepsilon}_P(\psi)$ is smaller than the left-hand side and therefore than the right-hand side, since, by Jensen's inequality, $|\psi(x)-P(\psi)|^{1+\varepsilon}\le \int_{\R^d}|\psi(x)-\psi(z)|^{1+\varepsilon}P(dz)$.
By symmetry, the same upper-bound holds for $ {\mathcal V}^{1+\varepsilon}_{\oP}(\opsi)$.\end{proof}

Now, we focus on the integrability of the translated transport map. This is an important technical point to give a sense to~\eqref{evowass} when the Markov process has jumps. Precisely, if $k(x,\dd z)$ denotes the jump kernel, we have to check the integrability of $\int_{\R^d}\int_{\R^d}|\psi(z)-\psi(x)|^\rho P(\dd x)k(x,\dd z)$, where $P$ is any time marginal of the process. For pure jump processes, we will be able to do this directly by taking advantage of the specific decomposition~\eqref{eq:marginale} of the time marginals. For PDMP, this is no longer possible and we typically would like to upper bound $\int_{\R^d}|\psi(x-y)-\psi(x)|^\rho P(\dd x) $ for different values of~$y\in \R^d$. In dimension~1,  this problem boils down from~\eqref{potential_1d} and the monotonicity of $z\mapsto T(z)$ to bound $\rho |y| \int_\R \max(|T(x-y)|^{\rho-1},|T(x)|^{\rho-1}) P(\dd x)$ from above. In view of the expression of the optimal transport map $T$, the finiteness of $\int_\R|T(x-y)|^{\rho-1}P(\dd x)$ is a condition 
intricately mixing the tails of $P$ and $\oP$. In the next proposition, we obtain a less intricate upper bound for this integral.

\begin{proposition}\label{prop_transport_1d}
Let $X$ and $\oX$ be two real random variables distributed according to $P$ and $\oP$. We note $F$ and $\tilde{F}$ the corresponding cumulative distribution functions,  $F^{-1}$ and $\tilde{F}^{-1}$ their right-continuous pseudo-inverse. We assume that $F$ is continuous and consider $T(x)=\tilde{F}^{-1}(F(x))$ the optimal transport map given by~\eqref{potential_1d}. We assume that for a given $y>0$, there is a function $\varphi_y\in L^{1}([0,1],\R_+)$ such that  
\begin{equation}\label{cond_phiy}\forall u\in (0,1), F(F^{-1}(u)+y)-u\leq  \int_0^u \varphi_y(v)\dd v. 
\end{equation}
Then, for any $q>0$, $\delta \in [0,+\infty]$, we have
\begin{equation}\label{eq: stima2TM} 
 \int_\R |T(x-y)|^q P(\dd x) \leq \E(|\oX|^q)+ \| |\oX|^q \|_{1+\frac{1}{\delta}} \| \varphi_y(U)  \|_{1+\delta},
\end{equation}
where $U$ is a uniform random variable on $[0,1]$, and the convention $1/0=+\infty$, $1/+\infty=0$. 
\end{proposition}
In~\eqref{eq: stima2TM}, we use the standard definition $\| Y\|_{\alpha}=\E[|Y|^\alpha]^{1/\alpha}$ for $\alpha>0$ and $\| Y\|_{\infty}=\inf \{m>0,\P(|Y|\le m)=1 \}$ for a real valued random variable~$Y$. Let us stress here that each side of~\eqref{eq: stima2TM} may be equal to $+\infty$, but the inequality always holds. The key assumption~\eqref{cond_phiy} implies in particular that $F^{-1}(0^+):=\lim_{u \rightarrow 0^+}F^{-1}(u)=-\infty$. If this was not the case, we would have $F(F^{-1}(0^+)+y)=0$ and thus $F^{-1}(0^+)\ge F^{-1}(0^+)+y$ which is contradictory. Last, let us mention that assuming $\varphi_y\ge 0$ is not restrictive: if~\eqref{cond_phiy} is satisfied by $\varphi_y\in L^{1}([0,1],\R)$, it is then also satisfied by $\max(0,\varphi_y)\in L^{1}([0,1],\R_+)$. If $X$ has a positive density $p(x)$ (i.e. $P(\dd x)=p(x)\dd x$), a natural choice is  $\varphi_y(u)=\max\left(\frac{p(F^{-1}(u)+y)}{p(F^{-1}(u))}-1,0\right)$. Heuristically, condition~\eqref{cond_phiy} is satisfied when $p(x)$ has an heavy 
left-
tail. For the  exponential density on $\R_-$ $p(x)=\ind_{\{x<0\}}\lambda e^{\lambda x}$ with $\lambda>0$, $\max\left(\frac{p(F^{-1}(u)+y)}{p(F^{-1}(u))}-1,0\right)=(e^{\lambda y}-1)\ind_{\{\frac{1}{\lambda}\log(u)+y<0\}}$ is bounded. For $p(x)=\ind_{\{x<-1\}}\frac{\alpha-1}{|x|^\alpha}$ with $\alpha>1$, $\varphi_y(u)=\ind_{\{1-yu^{\frac{1}{\alpha-1}}>u^{\frac{1}{\alpha-1}}\}}(1-yu^{\frac{1}{\alpha-1}})^{-\alpha} -1\le (1+y)^\alpha -1$ is again bounded.  

\begin{proof}

Since $T$ is a nondecreasing and $y>0$, we have
 \begin{align}
|T(x-y)|^q &\leq |T(x-y)|^q\ind_{\{T(x-y)<0\}} + |T(x)|^q\ind_{\{T(x-y) \ge 0\}}=(|T(x-y)|^q- |T(x)|^q)\ind_{\{T(x-y)<0\}}+ |T(x)|^q \nonumber \\
&\leq (|T(x-y)|^q- |\min(T(x),0)|^q)\ind_{\{T(x-y)<0\}}+ |T(x)|^q. \label{eq:lemmatm1}
 \end{align}
Let $x_0=\sup \{x \in \R, T(x) < 0 \}$ with convention $\sup \emptyset = -\infty$. We note $H_q(z)=-(-\min(T(z),0))^q$. This is a nondecreasing right-continuous function that induces a measure denoted by $\dd H_q(z)$. We have
\begin{align*}
 (|T(x-y)|^q- |\min(T(x),0)|^q)\ind_{\{T(x-y)<0\}}&=(H_q(x)-H_q(x-y))\ind_{\{T(x-y)<0\}}\\&=\ind_{\{T(x-y)<0\}} \int_{\R} \ind_{\{x-y<z\le x\}} \dd H_q(z).
\end{align*}
We note that $\{x-y < x_0 \} \subset \{T(x-y)<0 \} \subset \{x-y \le x_0 \}$. Since $F$ is continuous, and using the Fubini-Tonelli theorem, we obtain
\begin{align*}
\int_\R (|T(x-y)|^q- |\min(T(x),0)|^q)\ind_{\{T(x-y)<0\}} P(\dd x)&= \int_\R   \int_{\R}\ind_{\{x-y < x_0\}} \ind_{\{x-y<z\le x\}}  P(\dd x) \dd H_q(z) \\
&=\int_\R F(\min(x_0,z)+y)-F(z) \dd H_q(z) \\
&= \int_\R \ind_{\{z \le x_0\}}( F(z+y)-F(z)) \dd H_q(z), 
\end{align*}
since $\dd H_q((x_0,+\infty))=0$. We define $\bar{H}_q(u)=-(-\min(\tilde{F}^{-1}(u),0))^q$ for $u\in (0,1)$ and $\dd \bar{H}_q(u)$ the corresponding measure on $(0,1)$. By applying a change of variables in the Lebesgue-Stieltjes integral (see e.g. Proposition 4.10 p.~9 of~\cite{cf:RY}), we obtain since $z\le F^{-1}(F(z))$ and $\{z \le x_0\}\subset\{F(z) \le F(x_0)\}$
\begin{align*}
\int_\R \ind_{\{z \le x_0\}}( F(z+y)-F(z)) \dd H_q(z)&  \le \int_\R \ind_{\{F(z) \le F(x_0)\}}( F(F^{-1}(F(z))+y)-F(z)) \dd H_q(z) \\
&=\int_0^1   \ind_{\{ u \le F(x_0)\}} (F(F^{-1}(u)+y)-u )\dd \bar{H}_q(u) \\
&\le \int_0^1   \ind_{\{ u \le F(x_0)\}}  \int_0^u\varphi_y(v) \dd v \dd \bar{H}_q(u). 
\end{align*}
Since $\varphi_y\ge 0$, we have by Fubini-Tonelli theorem  that
\begin{align*}
\int_0^1 \int_0^1   \ind_{\{ u \le F(x_0)\}} \ind_{\{0<v<u<1\}}\varphi_y(v)\dd v \dd \bar{H}_q(u) &=\int_0^1 \int_0^1   \ind_{\{ v< u \le F(x_0)\}} \dd \bar{H}_q(u) \varphi_y(v)\dd v \\
&=\int_0^1  (\bar{H}_q(F(x_0))- \bar{H}_q(v))^+ \varphi_y(v) \dd v \\
&= \int_0^{F(x_0)} |\tilde{F}^{-1}(v)|^q \varphi_y(v) \dd v  \le  \| |\oX|^q \|_{1+\frac{1}{\delta}} \| \varphi_y(U)  \|_{1+\delta},
\end{align*}
by using $\bar{H}_q(F(x_0))=0$ and H\"older's inequality.
\end{proof}
\begin{remark}\label{rk_xtilde_bounded_below} Let us assume $F$ to be continuous. In the particular case where $\tilde{X}$ is bounded from below (i.e. $\tilde{F}^{-1}(0^+)>-\infty$), we get from the first inequality in~\eqref{eq:lemmatm1} that $|T(x-y)|^q \leq |\tilde{F}^{-1}(0^+)|^q + |T(x)|^q$, and therefore 
$$ \int_\R |T(x-y)|^q P(\dd x) \leq |\tilde{F}^{-1}(0^+)|^q + \E(|\oX|^q).$$
\end{remark}

 \section{Evolution of the Wasserstein distance between two pure jump Markov processes}
\label{sec:main}

 Let $\{X_t\}_{t\geq 0}$ and $\{\oX_t\}_{t\ge 0}$ be pure jump $\R^d$-valued Markov processes  with respective marginal laws $\{P_t\}_{t\ge 0}$ and $\{\oP_t\}_{t\ge 0}$ and
 infinitesimal generators given  by
 \begin{align}\label{eq: infinitesimalgen}
 L f(x)&= \lambda(x) \left(\int_{\R^d} k(x,\dd y)\left( f(y)-f(x)\right) \right)\\
 \oL f(x)&= \olambda(x) \left(\int_{\R^d} \ok(x,\dd y)\left( f(y)-f(x)\right) \right)\,.\end{align}
The jump intensities $\lambda$ and $\olambda :\R^d\to\R_+$ are measurable functions and $k$ and $\ok$ are Markov kernels (for all Borel subset $B$ of $\R^d$, $\R^d\ni x\mapsto k(x,B)$ is measurable and for all $x\in\R^d$, $k(x,.)$ is a probability measure on $\R^d$ endowed with its Borel sigma-field). To ease the trajectorial interpretation of $(\lambda,k)$ and avoid fake jumps, we suppose that  $\forall x\in\R^d$, $\lambda(x)k(x,\{x\})=0$. This is not restrictive since this condition may be achieved without modifying the generator $L$ by replacing $\lambda(x)$ by $\lambda(x)(1-k(x,\{x\}))$ and $k(x,.)$ by $k(x,.\cap\{x\}^c)/(1-k(x,\{x\}))$.
\begin{theorem}\label{teo:principale} Assume that 
$\bar{\lambda}:=\sup_{x\in\R^d}\max(\lambda(x),\olambda(x))<\infty$ and that $t\mapsto \E[|X_t|^{\rho(1+\varepsilon)}+|\oX_t|^{\rho(1+\varepsilon)}]$ is locally bounded for some $\varepsilon>0$. Let for each $t\ge 0$, $(\psi_t,\opsi_t)$ be a pair of Kantorovich potentials associated with $W_\rho(P_t,\oP_t)$. Then $t\mapsto \int_{\R^d} |L \psi_t(x)| P_t(\dd x) +\int_{\R^d}|\oL\, \opsi_t(x) |\oP_t(\dd x)$ is locally bounded on $(0,+\infty)$, $t\mapsto W_\rho^\rho(P_t,\oP_t)$ is locally Lipschitz on $(0,+\infty)$
  and for almost any $t \in(0,\infty)$
  \begin{equation}\label{eq: main result}
   \frac{\dd }{\dd t} W_\rho^{\rho}(P_t,\oP_t)= - \int_{\R^d} L \psi_t(x) P_t(\dd x) -\int_{\R^d}\oL\, \opsi_t(x) \oP_t(\dd x)\,.
  \end{equation}
Moreover, $t\mapsto \int_{\R^d} |L \psi_t(x)|P_t(\dd x)+\int_{\R^d} |L \opsi_t(x)|\oP_t(\dd x)$ is locally integrable on $[0,\infty)$ and for every $ t\geq 0$ the following integral formula holds true
  \begin{equation}\label{eq:formula integrale}
   W_\rho^{\rho}(P_t,\oP_t)= W_\rho^{\rho}(P_0,\oP_0)- 
   \int_0^t \left[ \int_{\R^d} L \psi_r(x) P_r(\dd x) +\int_{\R^d}\oL\, \opsi_r(x) \oP_r(\dd x )\right] \dd r\,.
  \end{equation}
 \end{theorem}
\begin{remark}
   According to Lemma \ref{lemma:momenti} below the finiteness of
$$\int_{\R^d}|x|^{\rho(1+\varepsilon)}(P_0(\dd x)+\oP_0(\dd x))+\sup_{x\in\R^d} \int_{\R^d}(k(x,\dd y)+\ok(x,\dd y))|y-x|^{\rho(1+\varepsilon)}$$ is sufficient to ensure that
$t\mapsto \E[|X_t|^{\rho(1+\varepsilon)}+|\oX_t|^{\rho(1+\varepsilon)}]$ is locally bounded.
\end{remark}
The next proposition, the proof of which is postponed to Section \ref{sec:integra}, and the next lemma respectively state the integrability properties and the integral formula for $\int_{\R^d}f(x)P_t(\dd x)-\int_{\R^d}f(x)P_s(\dd x)$ needed to establish Theorem \ref{teo:principale}.
\begin{proposition}\label{prop:mom}
   Under the assumptions of Theorem \ref{teo:principale}, for $\delta\in[0,\varepsilon)$, 
\begin{align*}
   &(t,r)\mapsto \int_{\R^d}|\psi_t(x)-P_t(\psi_t)|^{1+\delta}P_r(\dd x)\mbox{ is locally bounded on }(0,+\infty)\times [0,+\infty),\\&(t,r)\mapsto \int_{\R^d}\left(\int_{\R^d}\lambda(x) k(x,\dd y)|\psi_t(y)-\psi_t(x)|\right)^{1+\delta}P_r(\dd x)\mbox{ is locally bounded on }(0,+\infty)^2,\\
&t\mapsto \int_{\R^d\times \R^d}\lambda(x) k(x,\dd y)|\psi_t(y)-\psi_t(x)|P_t(\dd x)\mbox{ is locally integrable on }[0,+\infty)
.\end{align*}\end{proposition}
\begin{remark}
   For $r>0$, $P_0$ is absolutely continuous with respect to $P_r$ (see \eqref{eq:marginale} below) but, in general, $P_r$ is not absolutely continuous with respect to $P_0$ so that the finiteness of $\int_{\R^d}|\psi_0(x)-P_0(\psi_0)|^{1+\delta}P_r(\dd x)$ is not guaranted.
\end{remark}
 \begin{lemma}\label{lemma:dynkin}
  Let $f:\R^d\to\bar\R$ be a measurable function. If for $0\leq s\leq t$
  \begin{equation}\label{eq:cond dynkin}
   \int_s^t \int_{\R^d\times\R^d}\lambda(x) |f(y)-f(x)|k(x,\dd y)P_r(\dd x)\dd r< +\infty
  \end{equation}
then $f(X_t)-f(X_s)$ is integrable and
\begin{equation}\label{eq:form dynk}
 \E[f(X_t)-f(X_s)]=
 \int_s^t \int_{\R^d} Lf(x) P_r(\dd x)\dd r
 \end{equation}
 \end{lemma}
\begin{proof}[Proof of Lemma \ref{lemma:dynkin}]
 We can represent the process $\{X_t\}_{t\ge 0}$ as
 \begin{equation}\label{eq: Xt martingala}
  X_t=X_0 +\int_{(0,t]\times [0,1]\times [0,1]} \ind_{\{\lambda(X_{r-})\geq \bar\lambda u\}}\left(\varphi(X_{r-},v)-X_{r-}\right)
  N(\dd r, \dd u,\dd v)
 \end{equation}
where $N$ is a Poisson measure of intensity $\bar\lambda\dd r\dd u\dd v$ on $\R_+\times [0,1]\times[0,1]$ and for each $x\in\R^d$, the image of the uniform law on $[0,1]$ by $v\mapsto\varphi(x,v)$ is $k(x,.)$. 
 
 In particular we have for every measurable function $f$
 \begin{equation}\label{eq: fXt martingala}
  f(X_w)=f(X_s) +\int_{(s,w]\times [0,1]\times [0,1]} \ind_{\{\lambda(X_{r-})\geq \bar\lambda u\}}\left(f(\varphi(X_{r-},v))-f(X_{r-})\right)
  N(\dd r, \dd u,\dd v)
 \end{equation}
 and it is a well-known fact (see \cite{cf:IW} p62) that if 
 \begin{equation}\label{eq:condizione martingala f}
\E \left[  \int _{(s,t]\times [0,1]\times [0,1]} \ind_{\{\lambda(X_{r-})\geq \bar\lambda u\}}\left|f(\varphi(X_{r-},v))-f(X_{r-})\right| 
  \bar\lambda \, \dd r\, \dd u\, \dd v \right]\,< \infty
 \end{equation}
then the process
\begin{equation}\label{eq:def Mt}
 M_w=\int_{(s,w]\times [0,1]\times [0,1]} \ind_{\{\lambda(X_{r-})\geq \bar\lambda u\}}\left(f(\varphi(X_{r-},v))-f(X_{r-})\right)
  \left(N(\dd r, \dd u,\dd v)-\bar\lambda\,\dd r\,\dd u\,\dd v)\right)
\end{equation}
is a centered martingale on $(s,t]$ and in particular,
\begin{align*}
   \E[f(X_t)-f(X_s)]&=\E[M_t]+\int_s^t\E\left[\bar\lambda\int_0^1\int_0^1 \ind_{\{\lambda(X_{r})\geq \bar\lambda u\}}\left(f(\varphi(X_{r},v))-f(X_{r})\right)\dd u\,\dd v\right]\dd r\\
&=0+\int_s^t\E\left[\lambda(X_{r})\int_{\R^d}k(X_{r},\dd y)(f(y)-f(X_{r}))\right]\dd r
\\&=\int_s^t \int_{\R^d} Lf(x) P_r(\dd x)\dd r.
\end{align*}
We conclude the proof by remarking that, by a similar computation, the left-hand side of \eqref{eq:condizione martingala f} is equal to $\int_s^t \int_{\R^d\times\R^d}\lambda(x) |f(y)-f(x)|k(x,\dd y)P_r(\dd x)\dd r$. 
\end{proof}
We are now ready to prove the theorem by making the formal proof given in the introduction rigorous.
 \begin{proof}[Proof of Theorem \ref{teo:principale}]
Let $t>0$ and $h>0$. Since, by Proposition \ref{prop:mom}, $(\psi_t,\opsi_t)\in L^{1}(P_{t+h})\times L^1(\oP_{t+h})$, one has 
\begin{align*}
 \frac{1}{h}\left(W_\rho^{\rho}(P_{t+h},\oP_{t+h}) - W_\rho^{\rho}(P_t,\oP_t)\right) 
 \geq &-\frac 1h \int_{\R^d}\psi_t(x)(P_{t+h}(\dd x)-P_t(\dd x))\\&-\frac 1h \int_{\R^d}\opsi_t(x)(\oP_{t+h}(\dd x)-\oP_t(\dd x)),
 \end{align*}
In the same way, for $h<t$, $(\psi_t,\opsi_t)\in L^{1}(P_{t-h})\times L^1(\oP_{t-h})$ and
\begin{align*}\frac{1}{h}\left(W_\rho^{\rho}(P_{t},\oP_{t}) - W_\rho^{\rho}(P_{t-h},\oP_{t-h})\right) 
 \leq &-\frac 1h \int_{\R^d}\psi_t(x)(P_t(\dd x)-P_{t-h}(\dd x))\\&-\frac 1h \int_{\R^d}\opsi_t(x)(\oP_t(\dd x)-\oP_{t-h}(\dd x)).
\end{align*}
Since by Proposition \ref{prop:mom}, $r\mapsto \int_{\R^d}\lambda(x)|\psi_t(y)-\psi_t(x)|k(x,\dd y) P_t(\dd x)$ is locally bounded on $(0,+\infty)$, Lemma \ref{lemma:dynkin} permits to transform  the right-hand sides of the previous inequalities to obtain :
\begin{align*}
 \frac{1}{h}\left(W_\rho^{\rho}(P_{t+h},\oP_{t+h}) -W_\rho^{\rho}(P_t,\oP_t)\right) 
 \geq &-\frac 1h \int_{t}^{t+h}\left(\int_{\R^d}L\psi_t(x)P_r(\dd x)+\int_{\R^d}\oL\opsi_t(x)\oP_r(\dd x)\right)\dd r\\\frac{1}{h}\left(W_\rho^{\rho}(P_{t},\oP_{t}) - W_\rho^{\rho}(P_{t-h},\oP_{t-h})\right) 
 \leq &-\frac 1h \int_{t-h}^{t}\left(\int_{\R^d}L\psi_t(x)P_r(\dd x)+\int_{\R^d}\oL\opsi_t(x)\oP_r(\dd x)\right)\dd r.
\end{align*}
Since, again by Proposition \ref{prop:mom}, $(t,r)\mapsto \int_{\R^d}|L\psi_t(x)|^{1+\delta}P_r(\dd x)+\int_{\R^d}|\oL\opsi_t(x)|^{1+\delta}\oP_r(\dd x)$ is locally bounded on $(0,+\infty)\times (0,+\infty)$, one deduces that $s\mapsto W_\rho^{\rho}(P_s,\oP_s)$ is locally Lipschitz and therefore $\dd s$ a.e. differentiable on $(0,+\infty)$. If it is differentiable at $t$, the equality $$\frac{\dd }{\dd t} W_\rho^{\rho}(P_t,\oP_t)= - \int_{\R^d} L \psi_t(x) P_t(\dd x) -\int_{\R^d}\oL\, \opsi_t(x) \oP_t(\dd x)$$
is obtained by taking the limit $h\to 0$ in the above inequalities once we have checked that $r\mapsto\int_{\R^d}L\psi_t(x)P_r(\dd x)$ and $r\mapsto\int_{\R^d}\oL\opsi_t(x)\oP_r(\dd x)$ are continuous at $t$. It is enough to check continuity of the first function since the same arguments apply to the second one and we first deal with the right continuity.

For $r\ge t$ and $f:\R^d\to\R_+$ measurable, $$\E[f(X_r)\ind_{\{\forall s\in[t,r],X_s=X_t\}}]=\int_{\R^d}f(x)e^{-\lambda(x)(r-t)}P_t(\dd x).$$ Hence the measure $Q_{(t,r)}(\dd x)=P_r(\dd x)-e^{-\lambda(x)(r-t)}P_t(\dd x)$ is non-negative. Writing
\[
 \int_{\R^d}\!\!\! L\psi_t(x)P_r(\dd x)= \int_{\R^d}\!\!\! L\psi_t(x) 
 e^{-\lambda(x)(r-t)}P_t(\dd x)+ \int_{\R^d}\!\!\! L\psi_t(x)Q_{(t,r)}(\dd x)\,
\]
we observe that, by dominated convergence, the first integral in the right-hand side converges
to $\int_{\R^d} L\psi_t(x)P_t(\dd x)$ as $r\to t$. By H\"older's inequality and since $$Q_{(t,r)}(\R^d)= 1 -\int_{\R^d}e^{-\lambda (x)(r-t)}P_t(\dd x)\leq 1-e^{-\bar\lambda (r-t)},$$
\begin{align*}
   \int_{\R^d}\!\!\! |L\psi_t(x)|Q_{(t,r)}(\dd x)&\leq 
 \left(\int_{\R^d}\!\!\! |L\psi_t(x)|^{1+\delta}Q_{(t,r)}(\dd x)\right)^{\frac{1}{1+\delta}}
 \left(\int_{\R^d}Q_{(t,r)}(\dd x)\right)^{\frac{\delta}{1+\delta}}\\
&\leq \left(\int_{\R^d}\!\!\! |L\psi_t(x)|^{1+\delta}P_{r}(\dd x)\right)^{\frac{1}{1+\delta}}
 \left(1-e^{-\bar\lambda (r-t)}\right)^{\frac{\delta}{1+\delta}}.
\end{align*}
As $r\to t$, the first term is bounded, while the second tends to $0$. Therefore $$\lim_{r\to t^+}\int_{\R^d}L\psi_t(x)P_r(\dd x)=\int_{\R^d}L\psi_t(x)P_t(\dd x).$$ To check the left continuity, we now let $r\le t$ and write 
\[
 \int_{\R^d}\!\!\! L\psi_t(x)P_r(\dd x)= \int_{\R^d}\!\!\! L\psi_t(x) 
 e^{\lambda(x)(t-r)}P_t(\dd x)- \int_{\R^d}\!\!\! L\psi_t(x)e^{\lambda(x)(t-r)}Q_{(r,t)}(\dd x)\,
\]
where the first term in right-hand side converges to $\int_{\R^d}L\psi_t(x) 
 P_t(\dd x)$ as $r\to t$ by dominated convergence. Moreover
\begin{align*}
   \int_{\R^d}\!\!\! |L\psi_t(x)|e^{\lambda(x)(t-r)}Q_{(r,t)}(\dd x)&\leq 
 \left(\int_{\R^d}\!\!\! |L\psi_t(x)|^{1+\delta}e^{\lambda(x)(t-r)}P_t(\dd x)\right)^{\frac{1}{1+\delta}}
 \left(\int_{\R^d}e^{\bar\lambda (t-r)}Q_{(r,t)}(\dd x)\right)^{\frac{\delta}{1+\delta}}\\
&\leq \left(e^{\bar\lambda (t-r)}\int_{\R^d}\!\!\! |L\psi_t(x)|^{1+\delta}P_t(\dd x)\right)^{\frac{1}{1+\delta}}
 \left(e^{\bar\lambda (t-r)}-1\right)^{\frac{\delta}{1+\delta}}
\end{align*}
so that the left-hand side goes to $0$ as $r\to t$, which implies the desired left continuity property.

Since $r\mapsto W_\rho^\rho(P_r,\oP_r)$ is locally Lipschitz on $(0,+\infty)$ with $\dd r$ a.e., $\frac{\dd }{\dd r} W_\rho^{\rho}(P_r,\oP_r)= - \int_{\R^d} L \psi_r(x) P_r(\dd x) -\int_{\R^d}\oL\, \opsi_r(x) \oP_r(\dd x)\,$, 
\begin{equation}
   \forall 0<s\le t,\;W_\rho^{\rho}(P_t,\oP_t)= W_\rho^{\rho}(P_s,\oP_s)- 
   \int_s^t \left[ \int_{\R^d} L \psi_r(x) P_r(\dd x) +\int_{\R^d}\oL\, \opsi_r(x) \oP_r(\dd x )\right] \dd r\,.\label{evolst}
\end{equation}
By Proposition \ref{prop:mom}, $r\mapsto \int_{\R^d} |L \psi_r(x)|P_r(\dd x)+\int_{\R^d} |L \opsi_r(x)|\oP_r(\dd x)$ is locally integrable on $[0,\infty)$. Therefore, as $s\to 0$, the integral in the right-hand side of \eqref{evolst} converges to the integral from $0$ to $t$ by dominated convergence. Since, by the triangle inequality for the Wasserstein distance, $|W_\rho(P_s,\oP_s)-W_\rho(P_0,\oP_0)|\le W_\rho(P_s,P_0)+W_\rho(\oP_s,\oP_0)$, it is enough to check that $\lim_{s\to 0^+}W_\rho(P_s,P_0)=0$ to conclude that \eqref{eq: main result} holds. This follows from the inequalities \begin{align*}
   W^\rho_\rho(P_s,P_0)\le \E[|X_s-X_0|^\rho]&\le \left(2^{\rho(1+\varepsilon)-1}\E\left[|X_s|^{\rho(1+\varepsilon)}+|X_0|^{\rho(1+\varepsilon)}\right]\right)^{\frac{1}{1+\varepsilon}}\left({\mathbb P}(X_s\neq X_0)\right)^{\frac{\varepsilon}{1+\varepsilon}}\\
&\le 2^{\rho}\left(\sup_{r\in[0,s]}\E\left[|X_r|^{\rho(1+\varepsilon)}\right]\right)^{\frac{1}{1+\varepsilon}}\left( 1-e^{-\bar{\lambda} s}   \right)^{\frac{\varepsilon}{1+\varepsilon}}.
\end{align*}
 \end{proof}

\section{Application to birth and death processes}\label{sec:birth_and_death}

Let $P_0$ and $\tilde{P}_0$ be two probability measures on $\N$. In this section, we consider two pure jump  Markov processes $\{X_t\}_{t\ge 0}$ and $\{\oX\}_{t\ge 0}$ with respective initial marginal laws $P_0$ and $\tilde{P}_0$ and with the same generator
$$Lf(x)=\eta(x)(f(x+1)-f(x))+\nu(x)(f(x-1)-f(x)),\;x\in\N.$$
The functions $\eta,\nu:\N \rightarrow \R_+$ with $\nu(0)=0$ are respectively the birth and death rates.  Note that we are precisely in the framework of Section~\ref{sec:main} by setting $\lambda(x)=\eta(x)+\nu(x)$ and, when $\lambda(x)>0$, $k(x,\dd y)=\frac{\eta}{\lambda}(x)\delta_{x+1}(\dd y)+\frac{\nu}{\lambda}(x)\delta_{x-1}(\dd y)$, and we still denote by $\{P_t\}_{t\ge 0}$ and $\{\oP_t\}_{t\ge 0}$ the  marginal laws of the processes $\{X_t\}_{t\ge 0}$ and $\{\oX\}_{t\ge 0}$. We will assume in the whole section that
\begin{equation}\label{hyp_eta}\exists C>0, \forall x \in \N, \eta(x)\le C(1+x),
\end{equation}
which ensures that the processes are well-defined for any $t\ge0 $ by preventing accumulation of jumps. Of course this sufficient condition is not necessary and existence also holds under suitable balance conditions between the birth and death rates. But to obtain estimations of $W_\rho(P_t,\tilde{P}_t)$ for any $\rho\ge 1$ below, we will assume that $\eta$ is Lipschitz continuous so that the affine growth assumption \eqref{hyp_eta} is no longer restrictive. For $\rho\ge 1$, we have formally
\begin{align*}
   \frac{\dd }{\dd t}\E[X_t^\rho]&=\E\left[\eta(X_t)((X_t+1)^\rho-X_t^\rho)+\nu(X_t)(|X_t-1|^\rho-X_t^\rho)\right]\le C\E\left[(1+X_t)((X_t+1)^\rho-X_t^\rho)\right]\\&\le Cc_\rho (\E[X_t^\rho]+1),
\end{align*}
using \eqref{hyp_eta} and introducting the finite constant $c_\rho=\sup_{x\in\N}\frac{(1+x)[(1+x)^\rho-x^\rho]}{1+x^\rho}$. By a standard localization procedure, we deduce that
\begin{equation}
   \E[X_t^\rho]\le (\E[X_0^\rho]+1)e^{c_\rho Ct}-1.\label{estimom}
\end{equation}
We now define
$${\rm Lip}(\eta)=\sup_{x\in \N}|\eta(x+1)-\eta(x)|, \ {\rm Lip}(\nu)=\sup_{x\in \N}|\nu(x+1)-\nu(x)|,  $$
and $\kappa:=\inf_{x\in\N}(\eta(x)+\nu(x+1)-\eta(x+1)-\nu(x))$ the Wasserstein curvature introduced by \cite{cf:J07}.
\begin{proposition}\label{prop_bd}
  Let us assume that $\eta$ and $\nu$ are bounded functions. Let $\rho\ge 1$, $P_0$ and $\tilde{P}_0$ such that $\E[X_0^\rho+\oX_0^\rho]<\infty$. Then, there exists a constant $C_\rho\in[0,+\infty)$ such that for any $t\ge 0$,
 \begin{align}
   W_\rho^{\rho}(P_t,\oP_t)\le  W_\rho^{\rho}(P_0,\oP_0)e^{-\kappa\rho t}+C_\rho({\rm Lip}(\eta)+{\rm Lip}(\nu))\int_0^te^{\kappa\rho(r-t)}(W_1(P_r,\oP_r)+1_{\{\rho>2\}}W_{\rho-1}^{\rho-1}(P_r,\oP_r))\dd r.\label{estiint}
 \end{align}
 Besides, in the particular cases $\rho=1$ and $\rho\in(1,2]$, we can take $C_1=0$ and $C_\rho=1$. 
\end{proposition}
\begin{proof}  
The initial marginals $P_0$ and $\tilde{P}_0$ are two probability measures on $\N$. For $U$ uniformly distributed on $[0,1]$ and $r\ge 0$, let $F_r(x)=P_r((-\infty,x])$, $\tilde{F}_r(x)=\tilde{P}_r((-\infty,x])$ and $\pi_r$ denote the law of $(F_r^{-1}(U),\tilde{F}_r^{-1}(U))$ which is an optimal coupling between $P_r$ and $\tilde{P}_r$. By a slight abuse of notation, for $x,y\in\N$, we will denote $\pi_r(x,y)$ in place of $\pi_r(\{(x,y)\})$.
Let $\rho\ge 1$. Since the jump rates are bounded, we can apply Theorem~\ref{teo:principale}, and get for $t\ge 0$,
\begin{align*}
   W_\rho^{\rho}(P_t,\oP_t)= W_\rho^{\rho}(P_0,\oP_0)+
   \int_0^t &\sum_{x,y\in\N}\pi_r(x,y)\bigg(\eta(x)(\psi_r(x)-\psi_r(x+1))+\nu(x)(\psi_r(x)-\psi_r(x-1))\\
&+\eta(y)(\tilde{\psi}_r(y)-\tilde{\psi}_r(y+1))+\nu(y)(\tilde{\psi}_r(y)-\tilde{\psi}_r(y-1))\bigg) \dd r
  \end{align*}Setting $\alpha_r(x,y)=\psi_r(x)-\psi_r(x+1)+\rho(y-x)|x-y|^{\rho-2}$, $\tilde{\alpha}_r(x,y)=\tilde{\psi}_r(y)-\tilde{\psi}_r(y+1)+\rho(x-y)|x-y|^{\rho-2}$, $\beta_r(x,y)=\psi_r(x)-\psi_r(x-1)+\rho(x-y)|x-y|^{\rho-2}$ and $\tilde{\beta}_r(x,y)=\tilde{\psi}_r(y)-\tilde{\psi}_r(y-1)+\rho(y-x)|x-y|^{\rho-2}$ (where by convention $z|z|^{\rho-2}=0$ when $z=0$), one deduces that
\begin{align}
   W_\rho^{\rho}(P_t,\oP_t)&= W_\rho^{\rho}(P_0,\oP_0)+
   \int_0^t \sum_{x,y\in\N}\pi_r(x,y)\bigg(\rho(\eta(x)+\nu(y)-\nu(x)-\eta(y))(x-y)|x-y|^{\rho-2}\notag\\
&+\eta(x)\alpha_r(x,y)+\eta(y)\tilde{\alpha}_r(x,y)+\nu(x)\beta_r(x,y)+\nu(y)\tilde{\beta}_r(x,y)\bigg) \dd r.\label{evowbd}
  \end{align}
From the definition of the Wasserstein curvature, one has $$\forall x,y\in\N,\;(\eta(x)+\nu(y)-\nu(x)-\eta(y))(x-y)\le -\kappa(x-y)^2.$$
Since $\pi_r$ is a coupling between $P_r$ and $\tilde{P}_r$, one deduces that
\begin{equation}
   \sum_{x,y\in\N}\pi_r(x,y)\rho(\eta(x)+\nu(y)-\nu(x)-\eta(y))(x-y)|x-y|^{\rho-2}\le-\kappa\rho W_\rho^{\rho}(P_r,\oP_r).\label{majowscv}
\end{equation}
Let $(x,y)\in\N^2$ be such that $\pi_r(x,y)>0$. By the optimality of the coupling $\pi_r$ and \eqref{optipi}, $-\psi_r(x)-\tilde{\psi}_r(y)=|x-y|^\rho$. Moreover, for all $(z,w)\in\N^2$, $-\psi_r(z)-\tilde{\psi}_r(w)\le |z-w|^\rho$. This yields
\begin{align*}
  &\alpha_r(x,y)+\rho(x-y)|x-y|^{\rho-2}=\psi_r(x)+\tilde{\psi}_r(y)-(\psi_r(x+1)+\tilde{\psi}_r(y))\le |x+1-y|^\rho-|x-y|^\rho,\\
&\tilde{\alpha}_r(x,y)+\rho(y-x)|x-y|^{\rho-2}=\psi_r(x)+\tilde{\psi}_r(y)-(\psi_r(x)+\tilde{\psi}_r(y+1))\le |y+1-x|^\rho-|x-y|^\rho,\\
&\alpha_r(x,y)+\tilde{\alpha}_r(x,y)=\psi_r(x)+\tilde{\psi}_r(y)-(\psi_r(x+1)+\tilde{\psi}_r(y+1))\le |x-y|^\rho-|x-y|^\rho=0.
\end{align*}
For $\rho>1$, $\exists C_\rho<\infty,\;\forall z\in{\mathbb Z},\;|z+1|^\rho-|z|^\rho-\rho z|z|^{\rho-2}\le C_\rho(1+1_{\{\rho>2\}}|z|^{\rho-2})$ (notice that $C_\rho$ can be chosen equal to $1$ when $\rho\in(1,2]$), so that
\begin{equation*}
   \alpha_r(x,y)\vee \tilde\alpha_r(x,y)\le C_\rho(1+1_{\{\rho>2\}}|x-y|^{\rho-2})\;\mbox{and}\;\alpha_r(x,y)+\tilde{\alpha}_r(x,y)\le 0.
\end{equation*}
As a consequence, we obtain
\begin{align}
   \eta(x)\alpha_r(x,y)+\eta(y)\tilde{\alpha}_r(x,y)&=(\eta(x)-\eta(y))\alpha_r(x,y)+\eta(y)(\alpha_r(x,y)+\tilde{\alpha}_r(x,y))\notag\\&\le |\eta(x)-\eta(y)|\times C_\rho(1+1_{\{\rho>2\}}|x-y|^{\rho-2})+0\notag\\&\le C_\rho{\rm Lip}(\eta)(|x-y|+1_{\{\rho>2\}}|x-y|^{\rho-1}).\label{boral}
\end{align}
For $\rho=1$, $|z+1|-|z|-\rho z|z|^{-1}=1_{\{z=0\}}$ so that
\begin{equation*}
   \alpha_r(x,y)\vee \tilde\alpha_r(x,y)\le 1_{\{x=y\}},\;\alpha_r(x,y)+\tilde{\alpha}_r(x,y)\le 0,\;\eta(x)\alpha_r(x,y)+\eta(y)\tilde{\alpha}_r(x,y)\le 0
\end{equation*}and \eqref{boral} holds with $C_1=0$.

One checks by a symmetric reasoning that, for $\rho \ge 1$,
\begin{align}
   \nu(x)\beta_r(x,y)+\nu(y)\tilde{\beta}_r(x,y)\le C_\rho{\rm Lip}(\nu)(|x-y|+1_{\{\rho>2\}}|x-y|^{\rho-1})\label{borbet}.
\end{align}
Plugging \eqref{majowscv}, \eqref{boral} and \eqref{borbet} into \eqref{evowbd}, one obtains
\begin{equation*}
   W_\rho^{\rho}(P_t,\oP_t)\le  W_\rho^{\rho}(P_0,\oP_0)+\int_0^tC_\rho({\rm Lip}(\eta)+{\rm Lip}(\nu))(W_1(P_r,\oP_r)+1_{\{\rho>2\}}W_{\rho-1}^{\rho-1}(P_r,\oP_r))-\kappa \rho W_\rho^{\rho}(P_r,\oP_r)\dd r,
\end{equation*}
and then~\eqref{estiint} by Gronwall's Lemma. 
\end{proof}

\begin{theorem}\label{thm_bd} We assume that~\eqref{hyp_eta} holds, $\kappa>-\infty$ and $\E[X_0^\rho+\tilde{X}_0^\rho]<\infty$ for some $\rho \ge 1$. Then, we have $W_1(P_t,\oP_t)\le  W_1(P_0,\oP_0)e^{-\kappa t}$ and, if ${\rm Lip}(\eta)+{\rm Lip}(\nu)<\infty$, then~\eqref{estiint} still holds. 
In particular, when $\rho\in (1,2]$,
\begin{equation}\label{estim_rho12}\forall t\ge 0,\;W_\rho^{\rho}(P_t,\oP_t)\le  W_\rho^{\rho}(P_0,\oP_0)e^{-\kappa\rho t}+({\rm Lip}(\eta)+{\rm Lip}(\nu))W_1(P_0,\oP_0)\frac{e^{-\kappa t}-e^{-\kappa \rho t}}{\kappa(\rho-1)}.
\end{equation}
\end{theorem}
Notice that for $t\ge 0$, since $P_t$ and $\tilde{P}_t$ are supported in $\N$, $\rho\mapsto W_\rho^\rho(P_t,\tilde{P}_t)$ is non-decreasing. \begin{proof}
  For $N \in \N^*$, we consider the following approximation of the process $\{X_t\}_{t\ge 0}$ (resp. $\{\oX_t\}_{t\ge 0}$). We set $X^N_0=\min(X_0,N)$ and $X^N_t=X_t$ for $t\le \tau^N=\inf\{s \ge 0: X_s\ge N \}$ (resp. $\oX^N_0=\min(\oX_0,N)$ and $\oX^N_t=\oX_t$ for $t\le \tilde{\tau}^N=\inf\{s \ge 0:\oX_s\ge N \}$). Then, we sample $ \{X^N_t\}_{t\ge \tau^N}$ (resp. $\{\oX^N_t\}_{t\ge \tilde{\tau}^N}$) independently from $X$ (resp. $\oX$)  with the birth rate $\ind_{\{x< N\}}\eta(x)$ and death rate $\nu(x)$.
By construction,  $\{X^N_t\}_{t\ge 0}$ and $\{\oX^N_t\}_{t\ge 0}$ are Markov processes on $\{0,\dots,N\}$ with the same generator $\ind_{\{x< N\}}\eta(x)(f(x+1)-f(x))-\nu(x)(f(x-1)-f(x))$. From Proposition~\ref{prop_bd}, we get
\begin{align}
  W_\rho^{\rho}(P^N_t,\oP^N_t)\le & W_\rho^{\rho}(P^N_0,\oP^N_0)e^{-\kappa^N\rho t}\notag\\
  &+C_\rho({\rm Lip}(\eta)+{\rm Lip}(\nu))\int_0^te^{\kappa^N\rho(r-t)}(W_1(P^N_r,\oP^N_r)+1_{\{\rho>2\}}W_{\rho-1}^{\rho-1}(P^N_r,\oP^N_r))\dd r,\label{majokn}
\end{align}
with $\kappa^N=\inf_{x\in \{0,\dots,N-1\}}\eta(x)+\nu(x+1)-\eta(x+1)\ind_{\{x+1 \not= N\}}-\nu(x)$. We have for $N\ge 2$,
$$ \inf_{x\in \{0,\dots,N-1\}}\eta(x)+\nu(x+1)-\eta(x+1)-\nu(x) \le \kappa^N \le \inf_{x\in \{0,\dots,N-2\}}\eta(x)+\nu(x+1)-\eta(x+1)-\nu(x)$$
and thus $\kappa^N \rightarrow \kappa$ as $N\rightarrow + \infty$. It remains to show the convergence $W_{\rho'}(P^N_t,\oP^N_t)\underset{N\rightarrow + \infty}{\rightarrow} W_{\rho'}(P_t,\oP_t)$ for $\rho'\in [1,\rho]$ and the convergence of the integral in the right-hand side of \eqref{majokn}. Using Assumption~\eqref{hyp_eta} like in the derivation of \eqref{estimom}, we obtain bounds on the $\rho$-moments that are uniform in $N$ locally uniformly in time, so that by Theorem~6.9 of~\cite{cf:Villani} and Lebesgue's theorem, it is enough to show the weak convergence of $P^N_t$ to $P_t$ as $N\to\infty$. Since $X^N_t=X_t$ when $t\le \tau^N$ and $\tau^N\rightarrow +\infty$ a.s., we even have the stronger almost sure convergence of $X^N_t$ to $X_t$. 
\end{proof}
For $\rho=1$, Theorem~\ref{thm_bd} recovers the estimation $W_1(P_t,\oP_t)\le  W_1(P_0,\oP_0)e^{-\kappa t}$ obtained in \cite{cf:J07} by any trajectorial coupling between the processes with common jumps on the diagonal. The estimation \eqref{estim_rho12}  when $\rho\in (1,2]$ is new to our knowledge. By iterating the inductive inequality \eqref{estiint}, we can also get new upper bounds on $W_\rho^{\rho}(P_t,\oP_t)$ for any $\rho>2$. For example, in the case $\kappa>0$, we have $e^{-\kappa\rho t}\le e^{-\kappa t}$, and we get that  $W_\rho^{\rho}(P_t,\oP_t)\le A(\rho) e^{-kt}$, where, for $\rho>2$, $A(\rho)$ satisfies from~\eqref{estiint} the induction formula $A(\rho)\le  W_\rho^{\rho}(P_0,\oP_0)+ \frac{C_\rho L}{\kappa (\rho-1)}(W_1(P_0,\oP_0)+ A(\rho-1)) $ with $L={\rm Lip}(\eta)+{\rm Lip}(\nu)$. From~\eqref{estim_rho12}, we eventually obtain that for $\rho>2$, when $\E[X_0^\rho+\tilde{X}_0^\rho]<\infty$,
$$\forall t\ge 0,\;W_\rho^{\rho}(P_t,\oP_t)\le \left( \sum_{j=0}^{\lceil \rho - 2 \rceil} \pi_{j-1}^\rho W_{\rho-j}^{\rho-j}(P_0,\oP_0)+\left( \sum_{j=1}^{\lceil \rho -1 \rceil} \pi_{j-1}^\rho  \right) W_1(P_0,\oP_0) \right) e^{- \kappa t},$$
where $\pi_j^\rho=\prod_{i=0}^j \frac{C_{\rho-i}L}{\kappa(\rho-i-1)}$, $\pi_{-1}^\rho=1$ and for $x\in\R$, $\lceil x\rceil$ denotes the integer such that $x\le\lceil x\rceil<x+1$. In the particular case $P_0=\delta_x$ and $\oP_0=\delta_y$ with $x,y\in\N$ such that $x\not= y$, we deduce that $W_\rho^{\rho}(P_t,\oP_t) \le \left(1+2\sum_{j=1}^{\lceil \rho - 2 \rceil} \pi_{j-1}^\rho+\pi^\rho_{\lceil \rho - 2 \rceil}\right)|x-y|^\rho e^{- \kappa t}$, which is similar to~\eqref{wasscontrmod}.
\begin{remark} By integrating explicitly~\eqref{estiint}, we could get sharper bounds for $W_\rho^{\rho}(P_t,\oP_t)$ with tedious and cumbersome calculations. Besides, in the case $\rho \in \N^*$ with $\rho\ge 3$, we can get a better upper bound than~\eqref{estiint} by bounding $|z+1|^\rho-|z|^\rho-\rho z|z|^{\rho-2}$ from above by $\sum_{k=0}^{\rho -2}\binom{\rho}{k}|z|^k$ instead of $C_\rho(1+1_{\{\rho>2\}}|z|^{\rho-2})$. This would involve in the induction all the functions $r\mapsto W_k^{k}(P_r,\oP_r)$ for $k=1,\dots,\rho-1$. Note that for $\rho\in [1,2]$, we have used the sharpest bound for $|z+1|^\rho-|z|^\rho-\rho z|z|^{\rho-2}$. 
\end{remark}

\noindent {\bf Analogy with one dimensional diffusions with multiplicative noise.}\\

At least at a formal level, one can obtain very similar estimations for one dimensional diffusions processes with generator $Lf(x)=\frac{1}{2}a(x)f''(x)+b(x)f'(x)$ where $a:\R\to\R_+$ is supposed to be Lipchitz continuous and such that $\forall x,y\in\R, |\sqrt{a}(x)-\sqrt{a}(y)|\le \rho(|x-y|)$ for some function $\rho$ such that $\int_{0^+}\frac{dz}{\rho(z)}=\infty$ and $b:\R\to\R$ is locally bounded and satifies the monotonicity condition
\begin{equation}
   \forall x,y\in\R,\;{\rm sgn}(x-y)(b(x)-b(y))\le -\kappa|x-y|\mbox{ where }{\rm sgn}(z)=1_{\{z\ge 0\}}-1_{\{z<0\}}.\label{monot}
\end{equation}
Under the conditions on $a$, if $\dd X_t=\sqrt{a}(X_t)\dd W_t+b(X_t)\dd t$ and $\dd \tilde{X}_t=\sqrt{a}(\tilde{X}_t)\dd W_t+b(\tilde{X}_t)\dd t$ with $(W_t)_{t\ge 0}$ a Brownian motion independent from $(X_0,\tilde{X}_0)$ distributed according to the optimal coupling $\pi_0$, it is well-known that the local time at $0$ of $X_t-\tilde{X}_t$ vanishes (see for instance Corollary 3.4 p.~390, \cite{cf:RY}) so that, by the It\^o-Tanaka formula (see Theorem 1.2 p.~222, \cite{cf:RY}) and \eqref{monot},
$$\dd |X_t-\tilde{X}_t|={\rm sgn}(X_t-\tilde{X}_t)\dd (X_t-\tilde{X}_t)\le {\rm sgn}(X_t-\tilde{X}_t)(\sqrt{a}(X_t)-\sqrt{a}(\tilde{X}_t))\dd W_t-\kappa |X_t-\tilde{X}_t|\dd t.$$
When the expectation of the stochastic integral vanishes and $t\mapsto \E[|X_t-\tilde{X}_t|]$ is locally integrable (both properties implied by $\E[|X_0|+|\tilde{X}_0|]<\infty$ which ensures $t\mapsto \E[|X_t|+|\tilde{X}_t|]$ is locally bounded), one deduces that $\forall t\ge 0$, $W_1(P_t,\tilde{P}_t)\le e^{-\kappa t}W_1(P_0,\tilde{P}_0)$. Notice that for the stochastic differential equation where $\sqrt{a}(x)$ and $b(x)$ have been replaced by the continuous and bounded coefficients $\sqrt{a}(-m\vee x\wedge m)$ and $b(-m\vee x\wedge m)$, weak existence holds whereas trajectorial uniqueness is deduced from the It\^o-Tanaka formula so that existence of a unique strong solution follows by the Yamada-Watanabe theorem. With standard uniform in $m$ estimation of moments, this ensures existence and trajectorial uniqueness for the original stochastic differential equation.

For $\rho>0$, let us suppose that for all $r\ge 0$, the Kantorovich potential functions $\psi_r$ and $\tilde\psi_r$ are $C^2$ and that \eqref{evowass} holds. Then, this equation writes
\begin{align*}
  W_\rho^\rho(P_t,\tilde{P}_t)=&W_\rho^\rho(P_0,\tilde{P}_0)\\
  &-\int_0^t\left(\int_\R\frac{1}{2}a(x)\psi''_r(x)+b(x)\psi'_r(x)P_r(\dd x)+\int_{\mathbb{R}}\frac{1}{2}a(x)\tilde\psi''_r(x)+b(x)\tilde\psi'_r(x)\tilde P_r(\dd x)\right)\dd r.
\end{align*}
When $F_r$ and $\tilde{F}_r$ are continuous and increasing, one can consider the optimal transport map $T_r(x)={\tilde F}_r^{-1}(F_r(x))$ mentionned in the introduction and its inverse $\tilde{T}_r(x)={ F}_r^{-1}(\tilde F_r(x))$. The first and second order Euler optimality conditions in \eqref{optim1d} write
$$\psi'_r(x)=\rho (T_r(x)-x)|T_r(x)-x|^{\rho-2}\mbox{ and }\psi''_r(x)\ge \rho(\rho-1)|T_r(x)-x|^{\rho-2}(T_r'(x)-1).$$
By symmetry and since $\tilde{T}_r(T_r(x))=x$ which implies $\tilde{T}_r'(T_r(x))=\frac{1}{T'_r(x)}$,
$$\tilde{\psi}'_r(T_r(x))=\rho (x-T_r(x))|T_r(x)-x|^{\rho-2}\mbox{ and }\tilde{\psi}''_r(T_r(x))\ge\rho(\rho-1)|T_r(x)-x|^{\rho-2}\left(\frac{1}{T'_r(x)}-1\right).$$
Now, using that $\tilde{P}_r$ is the image of $P_r$ by $T_r$, then the first order optimality conditions and last \eqref{monot}, one obtains
\begin{align*}
   -\int_{\R}b(x)\psi'_r(x)P_r(\dd x)&-\int_{\R}b(x)\tilde\psi'_r(x)\tilde P_r(\dd x)=-\int_{\R}b(x)\psi'_r(x)+b(T_r(x))\tilde{\psi}'_r(T_r(x))P_r(\dd x)\\
&=\rho\int_{\R}(b(T_r(x))-b(x))(T_r(x)-x)|T_r(x)-x|^{\rho-2}P_r(\dd x)\\
&\le -\kappa \rho\int_{\R}|T_r(x)-x|^{\rho}P_r(\dd x)=-\kappa \rho W_\rho^\rho(P_r,\tilde{P}_r).
\end{align*}
On the other hand,  we observe that when $T_r'(x)\ge 1$,
\begin{align*}
  a(x)&(1-T'_r(x))+a(T_r(x))\left(1-\frac{1}{T'_r(x)}\right)=a(x)\left(2-T'_r(x)-\frac{1}{T'_r(x)}\right)+(a(T_r(x))-a(x))\left(1-\frac{1}{T'_r(x)}\right)\\&
\le -a(x)\left(\sqrt{T'_r(x)}-\frac{1}{\sqrt{T'_r(x)}}\right)^2+|a(T_r(x))-a(x)|\times \left|1-\frac{1}{T'_r(x)}\right|\le 0+{\rm Lip}(a)|T_r(x)-x|\times 1
\end{align*}
and the same estimation holds when $T_r'(x)\le 1$ by a symmetric reasoning. Therefore we have
\begin{align*}
   -\int_{\R}a(x)\psi''_r(x)P_r(\dd x)&-\int_{\R}a(x)\tilde\psi''_r(x)\tilde P_r(\dd x)\\&\le \rho(\rho-1)\int_{\R}|T_r(x)-x|^{\rho-2}\left(a(x)(1-T'_r(x))+a(T_r(x))\left(1-\frac{1}{T'_r(x)}\right)\right)P_r(\dd x)\\ &\le \rho(\rho-1){\rm Lip}(a)\int_{\R}|T_r(x)-x|^{\rho-1}P_r(\dd x).
\end{align*}
We conclude that, for $\rho\ge 2$,
$$\forall t\ge 0,\;W_\rho^\rho(P_t,\tilde{P}_t)\le W_\rho^\rho(P_0,\tilde{P}_0)-\int_0^t\frac{\rho(\rho-1)}{2}{\rm Lip}(a)W_{\rho-1}^{\rho-1}(P_r,\tilde{P}_r)-\kappa\rho W_\rho^\rho(P_r,\tilde{P}_r)\dd r.$$
In a future work, we plan to investigate whether the approximation by Euler-Maruyama schemes considered in \cite{cf:AJK} permits to justify rigorously this estimation.

\section{Integrability with respect to the marginals of a pure jump process with bounded intensity of jumps}
\label{sec:integra}
In the present section, we deal with the pure jump Markov process $\{X_t\}_{t\ge 0}$ with marginals $\{P_t\}_{t\ge 0}$ and infinitesimal generator $L$ given by \eqref{eq: infinitesimalgen} with $\sup_{x\in\R^d}\lambda(x)\le\bar{\lambda}<\infty$. Our final goal is to prove Proposition \ref{prop:mom}.
 
For every $t\geq 0$, the law of $X_t$ is given by
 \begin{equation}\label{eq:marginale}
  P_t(\dd x)=\sum_{n\ge 0}P_{n,t}(\dd x)
 \end{equation}
where  $P_{0,t}(\dd x_0)=e^{-\lambda(x_0) t}P_0(\dd x_0)$ and for $n\ge 1$, under the convention $t_0=0$,
\begin{equation}\label{eq:Pnt}
\begin{split}
 P_{n,t}(\dd x_n)=&\int_{0\le t_1\le t_2\le \hdots\le t_n\le t}\int_{(\R^d)^{n}}P_0(\dd x_0)
 \prod_{j=0}^{n-1}(\lambda(x_j)e^{-\lambda(x_j)(t_{j+1}-t_{j})}k(x_{j},\dd x_{j+1}))\\
 &\phantom{\int_{0\le t_1\le t_2\le \hdots\le t_n\le t}\int_{(\R^d)^{n-1}}p_0(x_0)
   \prod_{j=0}^{n-1}\lambda(x_j)e^{-\lambda(x_j)(t_{j+1}-t_{j})}
   }e^{-\lambda(x_n)(t-t_n)}\dd t_1\hdots \dd t_n\,.
 \end{split}
\end{equation}The measure $P_{n,t}$ has the following trajectorial interpretation : for any $f:\R^d\to\R_+$ measurable, 
$$\E[f(X_t)\ind_{\{\{X_s\}_{s\ge 0}\mbox{ undergoes $n$ jumps on }[0,t]\}}]=\int_{\R^d}f(x)P_{n,t}(\dd x).$$
The key reason why we can relate integrability with respect to $P_s$ and to $P_t$ for $s\neq t$ is that all the marginals $\{P_r\}_{r>0}$ are equivalent probability measures.
\begin{lemma}Let $t>0$ and  $s\ge 0$.
\begin{equation}
  \forall n\in\N,\;P_{n,s}(\dd x)\le e^{\bar{\lambda}(t-s)^+}\left(\frac{s}{t}\right)^nP_{n,t}(\dd x).\label{eq:relazione pns pnt}
\end{equation}
Moreover, for $f:\R^d\to\bar{\R}$ measurable and $\eta>0$ we have,\begin{equation}
   \int_{\R^d}|f(x)|P_s(\dd x)\leq \ind_{\{s\le t\}}e^{\bar\lambda (t-s)} \int_{\R^d}|f(x)|P_t(\dd x)+\ind_{\{s>t\}}e^{\frac{\eta \bar\lambda s }{1+\eta}(\frac{s}{t})^{\frac1\eta}
} \left(\int_{\R^d}|f(x)|^{1+\eta}P_t(\dd x)\right)^{\frac{1}{1+\eta}}.\label{eq:integptps}
  \end{equation}
\end{lemma}
\begin{remark}
 This lemma ensures that if $s\leq t$ then $L^{\alpha}(P_t)\subset L^{\alpha}(P_s)$ for every $\alpha \geq 1$. 
  Such an inclusion may be strict.  If, for example, $\{X_r\}_{r\ge 0}$ is a Poisson process with positive parameter and $f(x)=\frac{\Gamma(x+1)}{t^x}$ 
  for some fixed $t>0$, then $f\in L^1(P_s)$ for $s< t$ but $f\notin L^1(P_s)$ for $s\geq t$. On the other hand, for $\eta>0$, we have the inclusion $L^{\alpha(1+\eta)}(P_t) \subset L^\alpha(P_s)$ for all $s,t>0$.
\end{remark}
\begin{proof}
Equation \eqref{eq:relazione pns pnt} clearly holds for $s=0$. For $s>0$, the change of variables $(s_1,s_2,\hdots,s_n)=\frac{s}{t}(t_1,t_2,\hdots,t_n)$ leads to
\begin{equation*}
\begin{split}
 P_{n,t}(\dd x_n)= &\left(\frac{t}{s}\right)^n\int_{0\le s_1\le s_2\le \hdots\le s_n\le s}\int_{(\R^d)^{n}}P_0(\dd x_0)
 \prod_{j=0}^{n-1}(\lambda(x_j)e^{-\frac{t}{s}\lambda(x_j)(s_{j+1}-s_{j})}k(x_{j},\dd x_{j+1}))\\
 &\phantom{\int_{0\le t_1\le t_2\le \hdots\le t_n\le t}\int_{(\R^d)^{n-1}}p_0(x_0)
   \prod_{j=0}^{n-1}\lambda(x_j)}e^{-\frac{t}{s}\lambda(x_n)(s-s_n)}\dd s_1\hdots \dd s_n,
 \end{split}
\end{equation*}
where, by convention, $s_0=0$.

Since for $0\le s_1\le s_2\le \hdots\le s_n\le s$, $$\prod_{j=0}^{n-1}e^{-(\frac{t}{s}-1)\lambda(x_j)(s_{j+1}-s_{j})}e^{-(\frac{t}{s}-1)\lambda(x_n)(s-s_n)}\ge e^{-\frac{(t-s)^+}{s}\bar{\lambda}\left(\sum_{j=0}^{n-1}(s_{j+1}-s_{j})+(s-s_n)\right)}=e^{-\bar{\lambda}(t-s)^+},$$
we deduce \eqref{eq:relazione pns pnt}. As a consequence,
$$\int_{\R^d}|f(x)|P_s(\dd x)=\sum_{n\geq 0}\int_{\R^d}|f(x)|P_{n,s}(\dd x)\leq e^{\bar{\lambda}(t-s)^+}\sum_{n\geq 0}\left(\frac{s}{t}\right)^n\int_{\R^d}|f(x)|P_{n,t}(\dd x)$$
and \eqref{eq:integptps} follows easily when $s\le t$. For $s\ge t$, applying H\"older's inequality in the right-hand side, we deduce that
 \begin{equation*}
  \begin{split}
   \int_{\R^d}|f(x)|\, P_s(\dd x)&\leq \left(\sum_{n\geq 0}\left(\frac{s}{t}\right)^{n\frac{1+\eta}{\eta}}
   \int_{\R^d}P_{n,t}(\dd x)\right)^{\frac{\eta}{1+\eta}}
   \left(\sum_{n\geq 0}\int_{\R^d}|f(x)|^{1+\eta}P_{n,t}(\dd x)\right)^{\frac{1}{1+\eta}}.\end{split}
 \end{equation*}
 We can conclude since 
$P_{n,t}(\R^d)\leq \bar\lambda^n \frac{t^n}{n!}$
and
$\sum_{n\geq 0}\left(\frac{s}{t}\right)^{n\frac{1+\eta}{\eta}}\bar\lambda ^n\frac{t^n}{n!}
 \leq \exp\left(\bar\lambda s \left(\frac{s}{t}\right)^{\frac{1}{\eta}}\right)
$
.\end{proof}

In order to prove Proposition \ref{prop:mom}, we have to deal with integrability with respect to the measure $\int_{\R^d}\lambda(x)k(x,\dd y)P_t(\dd x)$. To this aim, we introduce $Q_0(\dd x_0)=P_0(\dd x_0)$ and for $n\ge 1$, 
\begin{equation}\label{eq:Qn}
Q_n(\dd x_n)=\int_{\R^d}Q_{n-1}(\dd x_{n-1})\lambda(x_{n-1}) k(x_{n-1},\dd x_{n})=\int_{(\R^d)^{n}}P_0(\dd x_0)\prod_{j=0}^{n-1}\lambda(x_j) k(x_{j},\dd x_{j+1}). 
\end{equation}
Notice that one easily checks by induction on $n$ that for all $n\in\N$, $Q_n(\R^d)\le \bar{\lambda}^n$.
For all $t>0$, $P_t$ is equivalent to the measure $\sum_{n\in\N}\frac{Q_n}{n!}$.
Indeed, since in the definition \eqref{eq:Pnt} of $P_{n,t}$, $$e^{-\bar{\lambda} t}\le e^{-\lambda(x_n)(t-t_n)}\prod_{j=0}^{n-1}e^{-\lambda(x_j)(t_{j+1}-t_{j})}\le 1$$ and $\int_{0\le t_1\le t_2\le \hdots\le t_n\le t}dt_1\hdots dt_n=\frac{t^n}{n!}$, the first statement in the following lemma  holds. The lemma also relates integrability with respect to $P_t$ and to $\int_{\R^d}\lambda(x)k(x,\dd y)P_t(\dd x)$.

\begin{lemma}\label{remark: rapporto integrabilita}
 For every $n\geq 0$, for every $t\geq 0$ one has
\begin{equation}\label{eq: relazione qn pnt}
  e^{-\bar\lambda t}\frac{t^n}{n!}Q_n(\dd x)\leq P_{n,t}(\dd x)\leq\frac{t^n}{n!}Q_n(\dd x)\,.
\end{equation}
Moreover, for $f:\R^d\to\bar{\R}$ measurable, $t>0$ and any $\eta>0$, 
\begin{align}
   &\int_{\R^d\times \R^d}\lambda(x)|f(y)|k(x,\dd y)P_t(\dd x)
\leq C_{\eta}(t) 
\left(\frac{\sum_{n\ge 1}\int_{\R^d}|f(x)|^{1+\eta}P_{n,t}(\dd x)}{t}\right)^{\frac{1}{1+\eta}}\label{intptpkt}\\
&\mbox{ where }C_{\eta}(t)=e^{\bar\lambda t}\left(\frac{e^{\bar \lambda t \left(e^\frac{1+\eta}{e\eta}-1\right)}-e^{-\bar\lambda t}}{t}
\right)^{\frac{\eta}{1+\eta}}.\label{defceta}
\end{align}\end{lemma}
\begin{remark}
 Notice that $\lim_{t\to 0} C_\eta(t)=\left(\bar\lambda e^{\frac{1+\eta}{e\eta}}\right)^{\frac{\eta}{1+\eta}}$. Moreover, if $0<t\le s$, by \eqref{eq:relazione pns pnt}, $P_{n,t}\le e^{\bar{\lambda}s}\left(\frac{t}{s}\right)^nP_{n,s}$ so that $\frac{1}{t}\sum_{n\ge 1}P_{n,t}\le \frac{e^{\bar{\lambda}s}}{s}\sum_{n\ge 1}\left(\frac{t}{s}\right)^{n-1}P_{n,s}\le\frac{e^{\bar{\lambda}s}}{s}P_s$. With \eqref{intptpkt} for $t>0$ and using $\int_{\R^d}k(x,\dd y)P_0(\dd x)=Q_1(\dd y)\le \frac{e^{\bar{\lambda} s}}{s}P_{1,s}(\dd y)$ for $t=0$, we deduce that if $f\in L^{1+\eta}(P_s)$ with $s,\eta>0$, then $t\mapsto \int_{\R^d\times \R^d}\lambda(x)|f(y)|k(x,\dd y)P_t(\dd x)$ is bounded on $[0,s]$.
\end{remark}
\begin{proof}
From \eqref{eq:Qn} and \eqref{eq: relazione qn pnt} we have
\begin{align*}
\int_{\R^d}\lambda(x_n)k(x_n,\dd x_{n+1})P_{n,t}(\dd x_n)\le \frac{t^n}{n!}\,Q_{n+1}(\dd x_{n+1})=\frac{n+1}{t}\frac{t^{n+1}}{(n+1)!}\,Q_{n+1}(\dd x_{n+1}).
\end{align*}
Therefore
\begin{equation}\label{eq: da k in poi}
\begin{split}
\int_{(\R^d)^2}&\lambda(x)|f(y)|k(x,\dd y)P_{t}(\dd x)=\sum_{n\ge 0}\int_{(\R^d)^2}\lambda(x_n)|f(x_{n+1})|k(x_n,\dd x_{n+1})P_{n,t}(\dd x_n)\\
 &\leq\sum_{n\geq 0} \frac{n+1}{t}\frac{t^{n+1}}{(n+1)!}\int_{\R^d}|f(x_{n+1})|Q_{n+1}(\dd x_{n+1})\\
 &=\frac{e^{\bar\lambda t}}{t}\sum_{n\geq 1}\frac{n \,t^n}{n!} e^{-\bar\lambda t}\int_{\R^d}|f(x_{n})|Q_n(\dd x_n)\\
 &=\frac{e^{\bar\lambda t}}{t}\sum_{n\geq 1} \bar \lambda^{\frac{n\eta}{1+\eta}} n 
 \left(\frac{t^n}{n!}e^{-\bar\lambda t}\right)^{\frac{\eta}{1+\eta}}  
 \bar \lambda^{-\frac{n\eta}{1+\eta}}
 \left(\frac{t^n}{n!} e^{-\bar\lambda t}\right)^{\frac{1}{1+\eta}} \int_{\R^d}|f(x_{n})|Q_n(\dd x_n)\\
&\leq \frac{e^{\bar\lambda t}}{t}
\left(\sum_{n\geq 1}\bar \lambda^n n^{\frac{1+\eta}{\eta}} \frac{t^n}{n!} e^{-\bar\lambda t}\right)^{\frac{\eta}{1+\eta}} 
\left(\sum_{n\geq 1} \bar \lambda^{-n\eta}\frac{t^n}{n!} e^{-\bar\lambda t} 
\left(\int_{\R^d}|f(x_{n})|Q_n(\dd x_n)\right)^{1+\eta} \right)^{\frac{1}{1+\eta}}\,,
\end{split}
\end{equation}
where we used H\"older's inequality for the second inequality.

We deduce from H\"older's inequality and the bound $Q_n(\R^d)\le \bar{\lambda}^n$ that each term in the last sum can be bounded from above by
$\frac{t^n}{n!} e^{-\bar\lambda t} \int_{\R^d}|f(x_{n})|^{1+\eta}Q_n(\dd x_n)$ 
so that, by \eqref{eq: relazione qn pnt},
\begin{equation}\label{eq: primo termine aurelien}
 \sum_{n\geq 1} \bar \lambda^{-n\eta}\frac{t^n}{n!} e^{-\bar\lambda t} 
\left(\int_{\R^d}|f(x_{n})|Q_n(\dd x_n)\right)^{1+\eta} \leq
\sum_{n\ge 1}\int_{\R^d}|f(x)|^{1+\eta}P_{n,t}(\dd x).
\end{equation}
Moreover, since for $\alpha\geq 0$ and $x>0$,  $x^\alpha\leq e^{\frac{\alpha x}{e}}$, we have, for every $n\geq 1$
\begin{equation}\label{eq: stima interessante n}
 n^{\frac{1+\eta}{\eta}}\leq C(\eta)^n\,, \quad \text{with}\quad C(\eta)=e^{\frac{1+\eta}{e\eta}}\,.
\end{equation}
We obtain, consequently, that
\begin{equation}\label{eq: secondo termine aurelien}
 \sum_{n\geq 1}\bar \lambda^n n^{\frac{1+\eta}{\eta}} \frac{t^n}{n!} e^{-\bar\lambda t}
 \leq \sum_{n\geq 1} \frac{(\bar \lambda t C(\eta))^n}{n!} e^{-\bar\lambda t}
 =e^{\bar \lambda t (C(\eta)-1)}-e^{-\bar\lambda t}\, .
\end{equation}
We obtain the statement by plugging \eqref{eq: primo termine aurelien} and \eqref{eq: secondo termine aurelien} 
into \eqref{eq: da k in poi}.
\end{proof}
We are now ready to prove Proposition \ref{prop:mom}.
\begin{proof}[Proof of  Proposition \ref{prop:mom}]
By Proposition \ref{prop:1epsilon}, for $t\ge 0$, $\psi_t\in L^{1+\varepsilon}(P_t)$ and
\begin{equation}
   \int_{\R^d}|\psi_t(x)-P_t(\psi_t)|^{1+\varepsilon}P_t(\dd x)\le 2^{\rho(1+\varepsilon)}\E\left[|X_t|^{\rho(1+\varepsilon)}+|\oX_t|^{\rho(1+\varepsilon)}\right],\label{eq:integpsitpt}
\end{equation}
where the right-hand side is assumed to be locally bounded for $t\in[0,+\infty)$. Applying \eqref{eq:integptps} with $f(x)=|\psi_t(x)-P_t(\psi_t)|^{1+\delta}$ and $\eta=\frac{\varepsilon-\delta}{1+\delta}$, we deduce the first statement.

For $r,t>0$, using H\"older's inequality, then a convexity inequality, we obtain
\begin{align}
   &\int_{\R^d}\left(\int_{\R^d}\lambda(x) k(x,\dd y)|\psi_t(y)-\psi_t(x)|\right)^{1+\delta}P_r(\dd x)\notag\\
&\le \bar{\lambda}^\delta\int_{\R^d\times\R^d}\lambda(x) k(x,\dd y)|\psi_t(y)-P_t(\psi_t)+P_t(\psi_t)-\psi_t(x)|^{1+\delta}P_r(\dd x)\notag\\
&\le (2\bar{\lambda})^\delta\left(\int_{\R^d\times\R^d}\lambda(x) k(x,\dd y)|\psi_t(y)-P_t(\psi_t)|^{1+\delta}P_r(\dd x)+\bar{\lambda}\int_{\R^d}|\psi_t(x)-P_t(\psi_t)|^{1+\delta}P_r(\dd x)\right)\label{intkptpr}
\end{align}
By \eqref{intptpkt} with $\eta=\frac{\varepsilon-\delta}{2(1+\delta)}$ and \eqref{eq:marginale}, $$\int_{\R^d\times\R^d}\lambda(x) k(x,\dd y)|\psi_t(y)-P_t(\psi_t)|^{1+\delta}P_r(\dd x)\le C_{\frac{\varepsilon-\delta}{2(1+\delta)}}(r)\left(\frac{1}{r}\int_{\R^d}|\psi_t(x)-P_t(\psi_t)|^{1+\frac{\delta+\varepsilon}{2}}P_r(\dd x)\right)^{\frac{2(1+\delta)}{2+\delta+\varepsilon}}$$
where the right-hand side is smaller than a function of $(r,t)$ locally bounded on $(0,+\infty)^2$ multiplied by $\left(\int_{\R^d}|\psi_t(x)-P_t(\psi_t)|^{1+\varepsilon}P_t(\dd x)\right)^{\frac{1+\delta}{1+\varepsilon}}$ according to \eqref{eq:integptps} with $\eta=\frac{\varepsilon-\delta}{2+\varepsilon+\delta}$ (and Jensen's inequality for $r\le t$). By \eqref{eq:integptps} with $\eta=\frac{\varepsilon-\delta}{1+\delta}$, $\int_{\R^d}|\psi_t(x)-P_t(\psi_t)|^{1+\delta}P_r(\dd x)$ is also smaller than $\left(\int_{\R^d}|\psi_t(x)-P_t(\psi_t)|^{1+\varepsilon}P_t(\dd x)\right)^{\frac{1+\delta}{1+\varepsilon}}$ multiplied by a function of $(r,t)$ locally bounded on $(0,+\infty)^2$. Plugging these bounds in \eqref{intkptpr} and using \eqref{eq:integpsitpt}, we conclude that $$(r,t)\mapsto \int_{\R^d}\left(\int_{\R^d}\lambda(x) k(x,\dd y)|\psi_t(y)-\psi_t(x)|\right)^{1+\delta}P_r(\dd x)$$ is locally bounded on $(0,+\infty)^2$.

Starting from \eqref{intkptpr} with $\delta=0$ and $r=t$, using  \eqref{intptpkt} with $\eta=\varepsilon$ and \eqref{eq:integpsitpt}, we obtain that
\begin{align*}
   \int_{\R^d\times \R^d}&\lambda(x) k(x,\dd y)|\psi_t(y)-\psi_t(x)|P_t(\dd x)\le 2^{\rho}\left(C_\varepsilon(t)t^{-\frac{1}{1+\varepsilon}}+\bar{\lambda}\right)\E^{\frac{1}{1+\varepsilon}}\left[|X_t|^{\rho(1+\varepsilon)}+|\oX_t|^{\rho(1+\varepsilon)}\right].
\end{align*}Since $t\mapsto C_\varepsilon(t)$ and $t\mapsto\E\left[|X_t|^{\rho(1+\varepsilon)}+|\oX_t|^{\rho(1+\varepsilon)}\right]$ are locally bounded on $[0,+\infty)$ and $t\mapsto t^{-\frac{1}{1+\varepsilon}}$ is locally integrable, we conclude that $t\mapsto \int_{\R^d\times \R^d}\lambda(x) k(x,\dd y)|\psi_t(y)-\psi_t(x)|P_t(\dd x)$ is locally integrable on $[0,+\infty)$.
\end{proof}

 We now give a sufficient condition for $t\mapsto \E[|X_t|^\alpha]$ to be locally bounded.

\begin{lemma}\label{lemma:momenti}
 Assume that $\sup_{x\in\R^d}\lambda(x)\le\bar{\lambda}<\infty$ and let $\alpha\ge 1$. If 
 \begin{equation}\label{eq:lemmaMomenti}
\bar k_{\alpha}^X=\max\left(\E[|X_0|^{\alpha}], \sup_x \int_{\R^d}k(x,\dd y) |y-x|^{\alpha} \right)<+\infty\,,  
 \end{equation}
then for every $t\geq 0$
\begin{equation}\label{eq: stimaMomenti}
\int_{\R^d}|x|^{\alpha}P_t(\dd x)\leq  \bar k_{\alpha}^X \left(
\sum_{n=0}^{\lceil\alpha \rceil-1}\frac{(n+1)^{\alpha}}{n!}(\bar\lambda t)^n+\frac{(\lceil\alpha\rceil+1)^{\alpha}}{\lceil\alpha\rceil!}
(\bar\lambda t)^{\lceil\alpha\rceil}e^{\bar\lambda t}\right),
\end{equation}
where $\lceil \alpha\rceil$ denotes the integer such that $\alpha\le \lceil \alpha\rceil<\alpha+1$.
\end{lemma}
\begin{proof}
 We observe that, for every $n\geq 0$
 \begin{equation}
 \begin{split}
 \int_{\R^d}|x_n|^{\alpha}P_{n,t}(\dd x_n) &\leq 
 \int_{0=t_0<t_1<\cdots<t_n<t}\int_{\R^d}
(n+1)^{\alpha-1} \left(|x_0|^{\alpha}+ |x_1-x_0|^{\alpha}+\cdots +|x_n-x_{n-1}|^{\alpha}\right) \\
 &\phantom{12345}P_0(\dd x_0)\prod_{j=0}^{n-1}\left(\lambda(x_j)e^{-\lambda(x_j)(t_{j+1}-t_j)}k(x_j,\dd x_{j+1})\right)e^{-\lambda(x_n)(t-t_n)}
 \dd t_0\cdots\dd t_n\\
 &\leq (n+1)^{\alpha-1}\frac{(\bar\lambda t)^n}{n!}(n+1)\bar k_{\alpha}^X
 =(\bar\lambda t)^n\frac{(n+1)^{\alpha}}{n!}\bar k_{\alpha}^X\,.
 \end{split}
  \end{equation}
  
 We now observe that, for $\alpha \geq 1$, if $n\geq \lceil\alpha\rceil$
 \begin{equation*}
  \frac{(n+1)^{\alpha}}{n!}\leq \frac{(n+1)^{\lceil\alpha\rceil}}{n(n-1)\ldots (n- \lceil\alpha\rceil +1)}\times \frac{1}{(n- \lceil\alpha\rceil)!}
  \leq 
  \frac{(\lceil\alpha\rceil+1)^{\lceil\alpha\rceil}}{\lceil\alpha\rceil!}\times \frac{1}{(n- \lceil\alpha\rceil)!}\,. 
\end{equation*}
Hence 
\begin{align*}
  \E[|X_t|^\alpha]&=\sum_{n\in\N}\int_{\R^d}|x_n|^\alpha P_{n,t}(\dd x_n)\\&\le \bar{k}^X_\alpha\left(\sum_{n=0}^{\lceil \alpha\rceil -1}\frac{(n+1)^{\alpha}}{n!}(\bar\lambda t)^n+\frac{(\lceil\alpha\rceil+1)^{\lceil\alpha\rceil}}{\lceil\alpha\rceil!}(\bar\lambda t)^{\lceil \alpha\rceil}\sum_{n\ge \lceil \alpha\rceil}\frac{(\bar\lambda t)^{n-\lceil \alpha\rceil}}{(n-\lceil \alpha\rceil)!}\right)\\&=\bar{k}^X_\alpha\left(\sum_{n=0}^{\lceil \alpha\rceil -1}\frac{(n+1)^{\alpha}}{n!}(\bar\lambda t)^n+\frac{(\lceil\alpha\rceil+1)^{\lceil\alpha\rceil}}{\lceil\alpha\rceil!}(\bar\lambda t)^{\lceil \alpha\rceil}e^{\bar\lambda t}\right).
\end{align*}
\end{proof}
\section{Evolution of the Wasserstein distance between two one-dimensional Piecewise Deterministic Markov Processes (PDMP)}\label{se:PDMP}

In this section, we are interested in proving the identity~\eqref{evowass} when $\{P_t\}_{t\ge 0}$ and $\{\oP_t\}_{t\ge 0}$  are the marginal laws of two PDMP $\{X_t\}_{t\geq 0}$ and $\{\oX_t\}_{t\ge 0}$. Let us recall that the infinitesimal generator of a PDMP is given by
 $$L f(x)= V(x) \nabla f (x) + \lambda(x) \left(\int_{\R^d} k(x,\dd y)\left( f(y)-f(x)\right) \right),$$
where $V:\R^d \rightarrow  \R^d $ is  a vector field, while $\lambda(x)$ and $ k(x,\dd y)$ respectively denote as before the jump intensity and the jump law. To give a sense to~\eqref{evowass}, we need at least to show that $\int_{\R^d} |L \psi_t(x)|P_t(\dd x)<\infty$. By the triangular inequality, it is sufficient to upper bound
$$\int_{\R^d} |V(x) \nabla  \psi_t(x)| P_t(\dd x) + \int_{\R^d}|\lambda(x) \psi_t(x)| P_t(\dd x)+ \int_{\R^d}\lambda(x) \left|\int_{\R^d} k(x,\dd y) \psi_t(y)\right|P_t(\dd x). $$
By using Proposition~\ref{prop:1epsilon}, and under suitable conditions on $V$, $\lambda$ and $k$ that ensure in particular the moment boundedness of the PDMP, we see that we can upper bound the second term. We can also upper bound the first one by using Theorem 6.2.4 of~\cite{cf:AGS} that gives $|\nabla \psi_t(x)|=\rho|T_t(x)-x|^{\rho-1}\le \rho 2^{\rho-1}(|T_t(x)|^{\rho-1}+|x|^{\rho-1})$, $P_t(\dd x)$-a.e., where $T_t$ is the optimal transport map from $P_t$ to $\tilde{P}_t$.  Thus, the main difficulty is to upper bound $\int_{\R^d}\lambda(x) |\int_{\R^d} k(x,\dd y) \psi_t(y)|P_t(\dd x)$. In the case of pure jump Markov processes, we are able to handle this term thanks to Lemma~\ref{remark: rapporto integrabilita} and the remarkable property~\eqref{eq: relazione qn pnt} on the marginal probability measures. For general PDMP, we no longer have this property, and our strategy here will be to estimate the difference between $\psi_t(y)$ and $\psi_t(x)$. To do so, we will work in dimension one ($d=1$) and use 
the 
explicit formulas~\eqref{potential_1d} for the Kantorovich potentials. 

\subsection{Main result}
We consider $\{X_t\}_{t\geq 0}$ and $\{\oX_t\}_{t\ge 0}$ two $\R$-valued PDMP with respective marginal laws $\{P_t\}_{t\ge 0}$ and $\{\oP_t\}_{t\ge 0}$ and
 infinitesimal generators given  by
 \begin{align}\label{eq:inf_gen_pdmp}
 L f(x)&= V(x)  f' (x) + \lambda(x) \left(\int_{\R^d} k(x,\dd y)\left( f(y)-f(x)\right) \right)\\
 \oL f(x)&= \tilde{V}(x) f' (x) + \olambda(x) \left(\int_{\R^d} \ok(x,\dd y)\left( f(y)-f(x)\right) \right)\,.\end{align}

We will work under the following assumptions for both processes. 
\begin{assumption}\label{ass:Vlambda}
(Assumption on $V$, $\lambda$).
 \begin{itemize}
  \item[(i)]The vector field $V$ is locally Lipschitz and bounded.
  \item[(ii)]The intensity of the jumps $\lambda(x)$ is a continuous function of~$x$,  bounded from above, and we denote $\bar \lambda= \sup_x\lambda(x) <\infty$.
 \end{itemize}
\end{assumption}

\begin{assumption}\label{ass:jumps}
(Assumption on the jump kernel $k(x,\dd y)$).
 \begin{itemize}
\item[(iii)] Jumps are bounded:  $\exists M>0, \ \forall x\in\R, \  \int_\R \ind_{\{|x-y|>M\}}k(x,\dd y)=0$.
\item[(iv)]$\int_{\R} \lambda(x) e^{-\frac{x^2}{2}}k(x,\,\dd y)\dd x$ 
admits a density with respect to the Lebesgue measure~$\dd y$.  
\item[(v)] $x \in \R \mapsto k(x,\dd y)$ is continuous for the weak convergence topology. 
 \end{itemize}
\end{assumption}
\noindent {\bf Assumption $\widehat{{\textbf{\ref*{ass:jumps}}}}$.} {\it Assumptions~\ref{ass:jumps} (iii), (v) and}
\begin{itemize}
\item[$\widehat{\textit{(iv)}}$] If $P(\{ x\})=0$ for all $x\in \R$, then for any $y\in\R$, $\int_{\R} \lambda(x) k(x,\,\{y\}) P(\dd x)=0$. 
 \end{itemize}
Let us note that Assumption~$\widehat{\textit{(iv)}}$ is satisfied if for any $y\in\R$, $\{x \in \R, \lambda(x) k(x,\,\{y\})>0 \}$ is countable.

Last, we will need in our analysis to work with marginal distributions that have some specific properties. For a probability measure $P(\dd x)$ on $\R$, we denote by $F(x)=P((-\infty,x])$ its cumulative distribution function and $\bar{F}(x)=1-F(x)$. For $\rho \ge 1$, we define the following subset of probability measures on~$\R$ 
\begin{align}
\widehat{\mathcal{P}}_\rho = &\{ P (\dd x), \ s.t. \ F \text{ is continuous increasing}, \  \int_\R |x|^{\rho}P(\dd x) <\infty, \label{def_Erhohat} \\  & \phantom{aaaaaa} \exists c,C>0, \forall y>0,  \ \sup_x\frac{F(x+y)}{F(x)} \le c e^{Cy},\ \sup_x\frac{\bar F(x-y)}{\bar F(x)} \leq c e^{Cy} \}, \nonumber \\
\mathcal{P}_\rho = &\{ P (\dd x)=p(x)\dd x, \ s.t. \ P\in\widehat{\mathcal{P}}_\rho \}.
\end{align}
In particular, a probability measure of $\mathcal{P}_\rho $ has a density that is almost everywhere positive.  The next lemma ensures some nice properties of the optimal transport map between two probability measures in $\widehat{\mathcal{P}}_\rho$.
\begin{lemma}\label{lem_moments_transport} Let $P$ and $\tilde{P}$ denote two probability measures on $\R$ with respective cumulative distribution functions $F$ and $\tilde{F}$.  We assume that $F$ is continuous and denote $T(x)=\tilde{F}^{-1}(F(x))$  the associated optimal transport map.
\begin{itemize}
\item[(i)] If  $\tilde{F}$ is increasing,  then $T$ is continuous and the Kantorovich potential $\psi$ given by~\eqref{potential_1d} is $C^1$.
\item[(ii)]  If $\forall y>0,  \ \sup_x\frac{F(x+y)}{F(x)} \le c e^{Cy},\ \sup_x\frac{\bar F(x-y)}{\bar F(x)} \leq c e^{Cy}$, then we have
$$\forall q>0, \ \forall y>0, \ \int_{\R} |T(x\pm y)|^q  P(\dd x) \le ce^{Cy} \int_{\R} |x|^q \tilde{P}(\dd x). $$
\end{itemize}
\end{lemma}
Notice that the existence of a positive density for $\tilde{P}$ is sufficient to ensure that $\tilde{F}$ is increasing. 
\begin{proof}
Point~$(i)$ is obvious since the hypothesis ensures that $\tilde{F}^{-1}$ is continuous. For Point~$(ii)$, note that we necessarily have $c\ge 1$. The result for $\int_{\R} |T(x- y)|^q  P(\dd x)$ is a corollary of Proposition~\ref{prop_transport_1d} with $\varphi_y(u)=ce^{Cy}-1$ and $\delta=+\infty$. The result for $\int_{\R} |T(x+ y)|^q  P(\dd x)$ is obtained by symmetry, looking at the optimal transport between the images of $P$ and $\tilde{P}$ by $x\mapsto -x$. 
\end{proof}

We now state the main result of this section.
\begin{theorem}\label{th: main PDMP 1D} Let $\rho>1$. 
 Suppose that both the processes $\{X_t\}_{t\ge 0}$ and $\{\oX\}_{t\ge 0}$  are real valued PDMP satisfying Assumptions~\ref{ass:Vlambda},~\ref{ass:jumps} with initial marginals $P_0,\oP_0 \in \mathcal{P}_{\rho(1+\varepsilon)}$ for some $\varepsilon>0$.
 Then for every $t\geq 0$
 \begin{equation}\label{eq: deri wass PDMPPJ}
  W_\rho^{\rho}(P_t,\oP_t)-W_\rho^{\rho}(P_0,\oP_0)=
  -\int_0^t\left(\int_\R L\,\psi_r(x)\, P_r(\dd x) +\int_\R \oL\,\opsi_r(x)\, \oP_r(\dd x) \right)\dd r\,.
 \end{equation}
This result remains valid if $P_0$ or $\oP_0 $ belong to the larger set~$\widehat{\mathcal{P}}_{\rho(1+\varepsilon)}$ and the corresponding processes satisfy Assumption~$\widehat{{\text{\ref*{ass:jumps}}}}$ instead of~\ref*{ass:jumps}.
 \end{theorem}
\noindent The assumption on the initial laws may at first sight seem rather restrictive. However, it is always possible to approximate them by using the following lemma.
\begin{lemma}\label{lem_approx_vi}
Let $\rho \ge 1$. Let $X_0,Z$ be independent real random variables such that $\E[|X_0|^\rho]<\infty$ and $Z$ follows the density $p_Z(x)=\frac{1}{2}e^{-|x|}$. Then, for any $\eta>0$, $X_0+\eta Z \in \mathcal{P}_\rho$.
\end{lemma}
\begin{proof}
We first observe that $p_Z(x)\dd x \in \mathcal{P}_\rho$, since $F_Z(x)=\ind_{\{x<0\}}\frac{1}{2}e^x+\ind_{\{x\ge 0\}}(1-\frac{1}{2}e^{-x})$ satisfies $F_Z(x+y)\le F_Z(x)e^y$ for any $x\in \R, y>0$. This is clear when $x+y\le 0$. When $x\le 0$ and $x+y\ge 0$, this is true since $1-\frac{1}{2}e^{-(x+y)}\le\frac{1}{2}e^y e^{-x} \iff e^{x}\le \cosh(y) $. When $0\le x$, we have 
 $1-\frac{1}{2}e^{-(x+y)}\le e^y (1-\frac{1}{2} e^{-x})\iff 1+e^{-x}\sinh(y)\le e^y $ which holds since $1+\sinh(y)\le e^y$. 

Clearly, $\E[|X_0+\eta Z|^\rho]<\infty$. Since $\P(X_0+\eta Z\le x)=\E[F_{Z}((x-X_0)/\eta)]$, this law satisfies again the estimates in~\eqref{def_Erhohat} with $c=1$ and $C=1/\eta$. Besides,
 the density of $X_0+\eta Z$ is equal to $\E[p_{Z}((x-X_0)/\eta)]/\eta >0$.
\end{proof}

\begin{corollary}
Let $P_0$ and $\oP_0$ be two probability laws on~$\R$ with  finite moments of order $\rho(1+\varepsilon)$ for some $\varepsilon>0$. For $\eta>0$, we  define $P^\eta_0 \in \mathcal{P}_{\rho(1+\varepsilon)}$ and (resp. $\oP^\eta_0\in \mathcal{P}_{\rho(1+\varepsilon)}$) as the law of $X_0+\eta Z$ (resp. $\tilde{X_0}+\eta Z$) where $X_0$ (resp. $\tilde{X_0}$) and $Z$ are independent and respectively distributed according to $P_0$, $\oP_0$ and the density $p_Z(x)=\frac{1}{2}e^{-|x|}$. We denote  $P^\eta_t$ (resp. $\oP^\eta_t$)  the marginal law at time~$t$ of the PDMP with generator $L$ (resp. $\tilde{L}$). Then,  under Assumptions~\ref{ass:Vlambda} and~\ref{ass:jumps}, we have
\begin{equation}\label{limite_regularisation_loi_init} W_\rho^{\rho}(P_t,\oP_t)-W_\rho^{\rho}(P_0,\oP_0)=
  - \lim_{\eta \rightarrow 0}\int_0^t\left(\int_\R L\,\psi^\eta_r(x)\, P^\eta_r(\dd x) +\int_\R \oL\,\opsi^\eta_r(x)\, \oP^\eta_r(\dd x) \right)\dd r, 
\end{equation}
where $(\psi^\eta_r,\opsi^\eta_r)$ are Kantorovich potentials for  $P^\eta_r$ and $\oP^\eta_r$.
\end{corollary}
\begin{proof}
Let $t\ge 0$. On the one hand, we have from Proposition~\ref{prop_outil} below uniform in $\eta\in [0,1]$ bounds on the moments of order $\rho(1+\varepsilon)$ of the probability measures $P^\eta_t$ and $\oP^\eta_t$. On the other hand, we have the weak convergence of  $P^\eta_t$ (resp. $\oP^\eta_t$) to  $P_t$ (resp.  $\oP_t$) when $\eta \rightarrow 0$. This can be obtained by looking at the corresponding martingale problems and using standard tightness arguments, exactly as in the proof of Proposition~\ref{prop: convskoro}. Theorem~6.9 of~\cite{cf:Villani} then gives $W_\rho^{\rho}(P^\eta_t,\oP^\eta_t)\underset{\eta \rightarrow 0}\rightarrow W_\rho^{\rho}(P_t,\oP_t)$. We finally get~\eqref{limite_regularisation_loi_init} by Theorem~\ref{th: main PDMP 1D}.
\end{proof}

\subsection{Proof of Theorem~\ref{th: main PDMP 1D}}

To prove Theorem~\ref{th: main PDMP 1D} for the processes $\{X_t\}_{t\ge 0}$ and $\{\tilde{X}_t\}_{t\ge 0}$, we will approximate them by pure jump processes in order to use Theorem~\ref{teo:principale}. To do so, we define, for 
 $\mu\ge 1$ the following jump kernel
  \begin{equation}\label{eq: Kmu}
   k^\mu(x,\dd y)=  \frac{1}{\mu + \lambda(x)}\left( \mu \delta_{\left\{\Phi\left(x;\,\frac 1\mu\right)\right\}}(\dd y)
   + \lambda(x) k(x,\dd y) \right)\,,
  \end{equation}
where $\Phi(x,s)$ is the flow of the differential equation driven by~$V$, i.e. the solution of
\begin{equation}\label{eq:ode}
\frac{\partial}{\partial t}  \Phi(x,t)= V(\Phi(x,t)), \ 
\Phi(x,0)=x, \quad \forall\, x \in \R \,.
\end{equation}
We define the process $\{X_t^\mu\}_{t\ge 0}$ as a pure jump Markov process with jump intensity $\mu+\lambda(x)$,   transition kernel $k^\mu(x,\dd y)$ and initial distribution $P_0$. It has the following infinitesimal generator
\begin{equation}\label{eq: Lmu}
 L^\mu f(x)= \mu \left[f\left(\Phi\left(x,\frac 1\mu\right)\right)-f(x)\right] 
 +\lambda(x)\int_{\R} (f(y)-f(x))k(x,\dd y)\,.
\end{equation} 
We note $P^\mu_t$ the probability law of $X^\mu_t$. From Proposition~\ref{prop: convskoro}, the process $\{X_t^\mu\}_{t\ge 0}$ weakly converges to the process $\{X_t\}_{t\ge 0}$ when $\mu \rightarrow +\infty$. By a slight abuse of notation, we will set $X^{+\infty}=X$, $L^{+\infty}=L$ and $P^{+\infty}_t=P_t$. The following remark recalls a usual way to sample the processes $X$ and $X^\mu$. 

\begin{remark}\label{rk:representation} At some point, it will be convenient to work with a pathwise representation of the processes $X^{\mu}$ and $X$. To do so, let us consider $N$ a Poisson process of parameter~$\bar{\lambda}$, $N^\mu$ a Poisson process of parameter~$\mu$ and $(U^1_i)_{i\ge 1}$,$(U^2_i)_{i\ge 1}$  two sequences of independent uniform random variables on~$[0,1]$, that are independent and independent  from the initial value~$X_0\sim P_0$. Let us set for $x\in \R$, $u \in (0,1)$
$$H(x,u)=\inf \{ z \ :  \   k(x,(-\infty,z]) > u \}.$$
Then, the infinitesimal generators of the following processes 
\begin{align}
X_t&=X_0+\int_0^t V(X_s) \dd s +  \int_0^t \ind_{\{U^2_{N_s}\le \lambda(X_{s^-})/\bar{\lambda} \}}H(X_{s^-},U^1_{N_s})\dd N_s, \label{repr_X}\\
X^{\mu}_t&= X_0+ \int_0^t \Phi\left(X^{\mu}_{s^-},\frac 1\mu\right) \dd N^\mu_s+ \int_0^t \ind_{\{U^2_{N_s}\le \lambda(X^{\mu}_{s^-})/\bar{\lambda} \}}H(X^{\mu}_{s^-},U^1_{N_s})\dd N_s, \label{repr_Xmu}
\end{align}
are respectively $L$ and $L^\mu$. 
\end{remark}

To prove Theorem~\ref{th: main PDMP 1D}, we start by a technical result stating that the marginal laws of the PDMP and its approximations stay in the set  $\mathcal{P}_\rho$ of probability measures. It also gives uniform constants in $\mu$ that will be useful to handle limits when $\mu \rightarrow +\infty$. Then, we prove two technical lemmas. The first one brings on the integrability of Kantorovich potentials for probability laws in~$\mathcal{P}_\rho$. The last lemma shows the local uniform convergence of the Kantorovich potential  given by~\eqref{potential_1d} between $P^\mu_t$ and $\oP^\mu_t$ toward the same Kantorovich potential between $P_t$ and $\oP_t$. From now on, we denote by $F_t$, $\tilde{F}_t$, $F^\mu_t$ and $\tilde{F}^\mu_t$ the respective cumulative distribution functions of $P_t$, $\oP_t$, $P^\mu_t$ and $\oP^\mu_t$.

\begin{proposition}\label{prop_outil}

\begin{itemize}
\item  Let us assume that Assumptions~\ref{ass:Vlambda} and~\ref{ass:jumps}~(iii) hold. Then, we have for any $t>0$, $q>0$,
\begin{equation}\label{bound_moments}
\sup_{\mu \in [1/t,+\infty]} \E[|X^\mu_t-X_0|^q]\le 2^{(q-1)^+} \left(\|V\|_\infty^q t^q (q/e)^q e^e+M^q   (q/e)^q \exp(\bar{\lambda} t (e-1))  \right).
\end{equation}
\item Let us assume that $P_0(\dd x)=p_0(x)\dd x$ with $p_0(x)>0$ $\dd x$-a.e. (resp. $F_0$ is continuous increasing) and Assumption~\ref{ass:jumps}~(iv) (resp. $\widehat{\text{(iv)}}$) hold. Then, for any $t\ge 0$, $P_t(\dd x)=p_t(x)\dd x$ and $P^\mu_t(\dd x)=p^\mu_t(x)\dd x$ with $p_t(x)p^\mu_t(x)>0$ $\dd x$-a.e. (resp. $F_t$ and  $F_t^\mu$ are continuous increasing).
\item Let us assume that Assumptions~\ref{ass:Vlambda} and \ref{ass:jumps}~(iii) hold. If $C_0,c_0>0$ denote constants such that 
$$\forall y>0,  \ \sup_x\frac{F_0(x+y)}{F_0(x)}\le c_0 e^{C_0y},\quad \sup_x\frac{\bar F_0(x-y)}{\bar F_0(x)} \leq c_0 e^{C_0y},$$
we have for any $\mu\in [1,+\infty]$,
\begin{equation}\label{majo_tails}\forall y>0,  \ \sup_x\frac{F^\mu_t(x+y)}{F^\mu_t(x)} \le c_t e^{C_0y},\quad \sup_x\frac{\bar F^\mu_t(x-y)}{\bar F^\mu_t(x)} \leq c_t e^{C_0y},
\end{equation}
with $c_t= c_0^2 \exp\left( t(e^{C_0\|V\|_\infty}-e^{-C_0\|V\|_\infty})+ \bar{\lambda} t (e^{C_0M}-e^{-C_0M}) \right) $.
\end{itemize}
\end{proposition}
\noindent Since $\E[|X^\mu_t-X_0|^\rho] \le 2^{\rho-1}(\E[|X^\mu_t|^\rho]+\E[|X_0|^\rho])$, we easily get the following corollary.
\begin{corollary}\label{cor_stability_Prho}
Let Assumptions~\ref{ass:Vlambda},~\ref{ass:jumps}~(iii) and (iv) (resp. $\widehat{\text{(iv)}}$) hold. Then, if $P_0 \in \mathcal{P}_\rho$ (resp. $P_0 \in \widehat{\mathcal{P}}_\rho$) we have that $$\forall t\ge 0, \mu \in [1 ,+\infty], \ P^\mu_t \in \mathcal{P}_\rho \ \text{(resp. }  P^\mu_t \in \widehat{\mathcal{P}}_\rho).$$ 
\end{corollary}
\begin{proof}[Proof of Proposition~\ref{prop_outil}]
{\it Moments.} From~\eqref{repr_Xmu}, we have $|X^\mu_t-X_0|\le \frac{\|V\|_\infty}{\mu}N^\mu_t+M N_t$ and thus
$\E[|X^\mu_t-X_0|^q]\le 2^{(q-1)^+} \left(\left(\frac{\|V\|_\infty}{\mu}\right)^q \E[(N^\mu_t)^q]+M^q \E[N_t^q] \right)$. Since $x^q\le(q/e)^q  e^x $ for $x\ge 0$, we get $\E[(N^\mu_t/(\mu t))^q]\le (q/e)^q \exp(\mu t (e^{1/(\mu t)}-1))$ and $\E[N_t^q]\le (q/e)^q \exp( \bar{\lambda} t (e-1))$. We thus obtain~\eqref{bound_moments} since $s(e^{1/s}-1)\le e$ for $s\ge 1$. From~\eqref{repr_X}, we have similarly $|X_t-X_0|\le \|V\|_\infty t+M N_t$ and we deduce~\eqref{bound_moments} for $\mu=+\infty$.

{\it Density (under Assumption~\ref{ass:jumps}~(iv)).} Let $\tilde{k}^\mu(x,\dd y)=\frac{\lambda(x)}{\mu + \bar{\lambda}}k(x,\dd y)+\frac{\mu}{\mu + \bar{\lambda}} \delta_{\Phi(x,1/\mu)}(\dd y)+\frac{\bar{\lambda}-\lambda(x)}{\mu + \bar{\lambda}}\delta_x(\dd y)$ and $p_0$ denotes the density of $P_0$. Then, similarly to~\eqref{eq:marginale} and~\eqref{eq:Pnt}, we have $P^{\mu}_t(\dd x)=\sum_{n\ge 0}P^\mu_{n,t}(\dd x)$ with
 $P^\mu_{0,t}(\dd x_0)=e^{-(\mu+\bar{\lambda}) t}p_0(x_0) \dd x_0$ and for $n\ge 1$, 
$$P^\mu_{n,t}(\dd x_n)=e^{-(\mu+\bar{\lambda}) t}\frac{(\mu+\bar{\lambda})^n t^n}{n!}\int_{\R^{n}}p_0(x_0)\dd x_0 \prod_{j=0}^{n-1}\tilde{k}^\mu(x_{j},\dd x_{j+1}).$$
By induction, it is sufficient to check that for a probability density $p$ on $\R$,
\begin{align*}\int_{\R} p(x_0)\dd x_0 \tilde{k}^\mu(x_{0},\dd x_{1})=& \frac{1}{\mu + \bar{\lambda}}\int_{\R}\lambda(x_0) p(x_0)\dd x_0k(x_0,\dd x_1)\\
& +\frac{\mu}{\mu + \bar{\lambda}} \int_{\R} p(x_0)  \dd x_0 \delta_{\Phi(x_0,1/\mu)}(\dd x_1)  + p(x_1) \frac{\bar{\lambda}-\lambda(x_1)}{\mu + \bar{\lambda}}\dd x_1
\end{align*}
is absolutely continuous with respect to $\dd x_1$. On the one hand,  $p(x_0)\lambda(x_0)k(x_0,\dd x_1)\dd x_0$ is absolutely continuous with respect to $ e^{-\frac{x_0^2}{2}}\lambda(x_0)k(x_0,\dd x_1)\dd x_0$. From Assumption~\eqref{ass:jumps}(iv), we get that for every Lebesgue negligible set $A\subset \R$, $\int_{\R}p(x_0)\lambda(x_0)\int_A k(x_0,\dd x_1)\dd x_0 =0$ and therefore $\int_{\R} p(x_0)\lambda(x_0)\dd x_0 k(x_0,\dd x_1)$ is absolutely continuous with respect to $\dd x_1$. On the other hand, for a bounded measurable function $f:\R\rightarrow \R$, we have
\begin{align*}\int_{\R} f(x_1)\int_{\R} p(x_0)  \delta_{\Phi(x_0,1/\mu)}(\dd x_1) \dd x_0&= \int_{\R} f(\Phi(x_0,1/\mu)) p(x_0) \dd x_0.
\end{align*}
For any $t>0$, the flow $x\mapsto \Phi(x,t)$ is bijective continuous and increasing, and its inverse function is $x\mapsto \Phi(x,-t)$. We get (see for example Proposition 4.10 p.~9 of~\cite{cf:RY})
$$\int_{\R} f(\Phi(x_0,1/\mu)) p(x_0) \dd x_0=\int_{\R} f(x') p(\Phi(x',-1/\mu))  \dd \Phi(x',-1/\mu).$$
Since $V$ is bounded and locally Lipschitz, $x\mapsto \Phi(x,-t)$ is locally Lipschitz, which gives that it is almost everywhere differentiable and $\dd \Phi(x',-1/\mu)=\partial_x\Phi(x',-t) \dd x'$. Last, the inequality $P^\mu_{t}(\dd x)\ge P^\mu_{0,t}(\dd x)$ gives the almost everywhere positivity of the density.

For $X_t$ (which corresponds to $\mu=\infty$), we have similarly $P_t(\dd x)=\sum_{n\ge 0}P_{n,t}(\dd x)$ with  $P_{0,t}(\dd x_0)=e^{- \bar{\lambda} t} p_0(\Phi(x_0,-t)) \partial_x\Phi(x_0,-t) \dd x_0$ and for $n\ge 1$, 
\begin{align*}
 P_{n,t}(\dd x_n)=&e^{-\bar{\lambda} t}\bar{\lambda}^n\int_{0\le t_1\le t_2\le \hdots\le t_n\le t}p_{t_1,\dots,t_n}(\Phi(x_n,t_n-t)) \partial_x\Phi(x_n,t_n-t) \dd t_1\hdots \dd t_n  \dd x_n ,\\
p_{t_1,\dots,t_n}(x_n) \dd x_n:=&\int_{\R^{n}}p_0(x_0) \dd x_0
 \prod_{j=0}^{n-1} \tilde{k}(t_{j+1}-t_{j},x_{j},\dd x_{j+1})
 \end{align*}
with $t_0=0$ and $\tilde{k}(t,x,\dd y)=\frac{\lambda(\Phi(x,t))}{\bar{\lambda}}k(\Phi(x,t),\dd y) +\left(1-\frac{\lambda(\Phi(x,t))}{\bar{\lambda}}\right)\delta_{\Phi(x,t)}(\dd y)$.
This representation can be obtained from~\eqref{repr_X} by observing that $\int_{\R^{n}}p_0(x_0) \dd x_0
 \prod_{j=0}^{n-1} \tilde{k}(t_{j+1}-t_{j},x_{j},\dd x_{j+1})$ is the law of $X_{T_n}$ knowing that $\{N_t=n,T_1=t_1,\dots,T_n=t_n\}$ where $T_i=\inf\{t\ge 0, N_t=i \}$. By induction on $n$, we check that this law is absolutely continuous with respect to $\dd x_n$, which makes the definition of $p_{t_1,\dots,t_n}$ valid. It is sufficient to check this for $n=1$, and we have
$$\int_{\R} p_0(x_0)\dd x_0 \lambda(\Phi(x_0,t))k(\Phi(x_0,t),\dd x_1)=\int_{\R}p_0(\Phi(x',-t)) \lambda(x') \partial_x\Phi(x',-t)k(x',\dd x_1) \dd x',$$ which is absolutely continuous with respect to $\dd x_1$ by Assumption~\eqref{ass:jumps}(iv). By using the flow, we get that $\int_{\R} p_0(x_0)\dd x_0\left(1-\frac{\lambda(\Phi(x_0,t))}{\bar{\lambda}}\right)\delta_{\Phi(x_0,t)}(\dd x_1) $ is also absolutely continuous with respect to $\dd x_1$. Therefore, the density $p_{t_1,\dots,t_n}(x_n)$ exists. It is then clear that  $P_{n,t}(\dd x_n)=\P(X_t\in \dd x_n,N_t=n)$ admits the above density. Last, $P_{t}(\dd x)\ge P_{0,t}(\dd x)$ gives the positivity of the density since $\partial_x\Phi(x,-t)$ is almost everywhere positive.

{\it Continuous increasing cumulative distribution functions (under Assumption~$\widehat{\text{(iv)}}$).} We follow the same steps. We use that
$P^{\mu}_t(\dd x)=\sum_{n\ge 0}P^\mu_{n,t}(\dd x)$ with
 $P^\mu_{0,t}(\dd x_0)=e^{-(\mu+\bar{\lambda}) t}P_0(\dd x_0)$ and for $n\ge 1$, 
$$P^\mu_{n,t}(\dd x_n)=e^{-(\mu+\bar{\lambda}) t}\frac{(\mu+\bar{\lambda})^n t^n}{n!}\int_{\R^{n}}P_0(\dd x_0) \prod_{j=0}^{n-1}\tilde{k}^\mu(x_{j},\dd x_{j+1}).$$
Let  $P$ be a probability measure on~$\R$ such that $P(\{x\})=0$ for any $x\in \R$. We have for any $y\in \R$
\begin{align*}
\int_{\R} P(\dd x) \tilde{k}^\mu(x, \{y\})=&\frac{1}{\mu + \bar{\lambda}}\int_{\R}\lambda(x) P(\dd x)  k(x,\{y \})\\
& +\frac{\mu}{\mu + \bar{\lambda}}  P(\{\Phi(y,-1/\mu)\} )  +  \frac{\bar{\lambda}-\lambda(x_1)}{\mu + \bar{\lambda}} P(\{y\} )=0,
\end{align*}
by Assumption~$\widehat{\textit{(iv)}}$. By induction, for any $n\in\N$, $P^\mu_{n,t}$ does not weigh points. Therefore, $F^{\mu}_t$ is continuous. It is also increasing since $x \mapsto P^\mu_{0,t}((-\infty,x])$ is increasing. 

We can do similar calculations for $\{X\}_{t\ge 0}$, and show that for any $n\in\N^*$, $0\le t_1\le \dots \le t_n\le t$, $P_{t_1,\dots,t_n}(\dd x_n)=\int_{\R^{n}}P_0( \dd x_0)
 \prod_{j=0}^{n-1} \tilde{k}(t_{j+1}-t_{j},x_{j},\dd x_{j+1})$ does not weigh points. 
Therefore,
$x_n\mapsto P_{n,t}((-\infty,x_n])=e^{-\bar{\lambda} t}\bar{\lambda}^n\int_{0\le t_1\le t_2\le \hdots\le t_n\le t}  P_{t_1,\dots,t_n}((-\infty,\Phi(x_n,t_n-t)]) \dd t_1\hdots \dd t_n  $ is continuous. Since $x\mapsto P_{0,t}((-\infty,x])=e^{-\bar{\lambda} t} P_0((-\infty,\Phi(x,-t)])$ is continuous increasing and  $P_t(\dd x)=\sum_{n\ge 0}P_{n,t}(\dd x)$, we get that $F_t$ is  continuous increasing.

{\it Tails.} From~\eqref{repr_Xmu}, we have $X^\mu_t\ge X_0- \|V\|_\infty N^\mu_t/\mu-M N_t$. Since $P_0\in \mathcal{P}_\rho$ and using the independence, we get 
\begin{align*}
F^\mu_t(x+y)& \le \E\left[ F_0\left(x+y+{\|V\|_\infty}N^\mu_t/\mu+M N_t\right) \right] \\
& \le c_0 F_0(x) e^{C_0 y } \exp\left(\mu t(e^{C_0\frac{\|V\|_\infty}{\mu}}-1)+ \bar{\lambda} t (e^{C_0M}-1) \right)
\end{align*}
Looking at the right tail, we also get
\begin{align*}
\bar{F}^\mu_t(x)&\ge \E\left[ \bar{F}_0(x+ \|V\|_\infty N^\mu_t/\mu +M N_t) \right]\ge \frac{1}{c_0} \bar{F}_0(x) \exp\left(\mu t(e^{-C_0\frac{\|V\|_\infty}{\mu}}-1)+ \bar{\lambda} t (e^{-C_0M}-1) \right).
\end{align*}
Similarly, the other inequality $X^\mu_t\le X_0+\|V\|_\infty N^\mu_t/\mu+M N_t$ obtained from~\eqref{repr_Xmu} leads to
\begin{align*}
F^\mu_t(x)&\ge  \E\left[ F_0(x-\|V\|_\infty N^\mu_t/\mu -M N_t) \right]
\ge \frac{1}{c_0} F_0(x) \exp\left(\mu t(e^{-C_0\frac{\|V\|_\infty}{\mu}}-1)+ \bar{\lambda} t (e^{-C_0M}-1) \right), \\
\bar{F}^\mu_t(x-y)&\le \E\left[ \bar{F}_0(x-y-\|V\|_\infty N^\mu_t/ \mu-M N_t) \right] \\
&\le  c_0 \bar{F}_0(x)e^{C_0 y} \exp\left(\mu t(e^{C_0\frac{\|V\|_\infty}{\mu}}-1)+ \bar{\lambda} t (e^{C_0M}-1) \right)  .
\end{align*}
Therefore, we obtain that both $\frac{F^\mu_t(x+y)}{F^\mu_t(x)}$ and $\frac{\bar{F}^\mu_t(x-y)}{\bar{F}^\mu_t(x)}$ are upper bounded by
$$ c_0^2 \exp\left(\mu t(e^{C_0\frac{\|V\|_\infty}{\mu}}-e^{-C_0\frac{\|V\|_\infty}{\mu}})+ \bar{\lambda} t (e^{C_0M}-e^{-C_0M}) \right) e^{C_0 y }. $$
From~\eqref{repr_X}, we check by similar calculations that the upper bound is also true for $\mu=+\infty$ with $\mu t(e^{C_0\frac{\|V\|_\infty}{\mu}}-e^{-C_0\frac{\|V\|_\infty}{\mu}})=2C_0\|V\|_\infty t$. Since $[1,+\infty] \ni \mu \mapsto \mu t(e^{C_0\frac{\|V\|_\infty}{\mu}}-e^{-C_0\frac{\|V\|_\infty}{\mu}})$ is nonincreasing, we get~\eqref{majo_tails}. 
\end{proof}

\begin{lemma}
Let $\rho>1$. We assume that Assumption~\ref{ass:Vlambda} holds and that $P,\tilde{P}\in \widehat{\mathcal{P}}_{\rho \alpha}$ for some $\alpha \ge 1$.   Let $\psi\in C^1$ be such that $\psi'(x)=\rho|x-T(x)|^{\rho-2}(T(x)-x)$  with $T(x)=\tilde{F}^{-1}(F(x))$ (i.e. this is a Kantorovich potential associated to the $\rho$-Wasserstein distance from~\eqref{potential_1d}). Let $c,C>0$ denote constants such that~\eqref{def_Erhohat} holds for $P$. 
Then, for any $\mu\in [1,+\infty]$, we have 
\begin{align}\label{majo_Lmu}
\int_{\R} |L^\mu \psi(x)|^{\alpha}P( \dd x) \, \le&\, 2^{\alpha-1}\|V\|_\infty^\alpha \rho^{\alpha}   2^{(\alpha(\rho-1)-1)^+} \Bigg( 2 c e^{C\|V\|_\infty} \int_{\R}|x|^{\alpha(\rho-1)} \tilde{P}(\dd x) \\ & \phantom{2^{\alpha-1} \rho^{\alpha}   } +2^{(\alpha(\rho-1)-1)^+} \left(\int_{\R}|x|^{\alpha(\rho-1)} P(\dd x) +\|V\|_\infty^{\alpha(\rho-1)} \right) \Bigg) \nonumber \\
&+2^{\alpha-1}\bar{\lambda}^\alpha 2^{\alpha \rho -1} \Bigg( (1+2^{\alpha \rho -1})\int_{\R} |x|^{\alpha \rho} P(\dd x)+ 2^{\alpha \rho -1} M^{\alpha \rho} \nonumber\\
&\phantom{+2^{\alpha-1}\bar{\lambda}^\alpha 2^{\alpha \rho -1}()} +  (1+ 2ce^{CM})\int_{\R} |x|^{\alpha \rho} \tilde{P}(\dd x) \Bigg).\nonumber
\end{align}
\end{lemma}
\begin{proof}
 We have $\psi(\Phi(x,1/\mu))-\psi(x)=\int_0^{1/\mu} \psi'(\Phi(x,s))V(\Phi(x,s)) \dd s$ and thus
 $$|L^\mu \psi(x)|\le \mu \|V\|_\infty \int_0^{1/\mu} | \psi'(\Phi(x,s))| \dd s+\bar{\lambda} \int_{\R} |\psi(y)-\psi(x)|k(x,\dd y).$$
Also, we obtain $|L^{+\infty} \psi(x)|\le \|V\|_\infty |\psi'(x)|+\bar{\lambda} \int_{\R} |\psi(y)-\psi(x)|k(x,\dd y)$, which leads to
$$\forall \mu \in [1,+\infty], |L^\mu \psi(x)|\le \|V\|_\infty \max_{|y|\le \|V\|_\infty}|\psi'(x+y)|+\bar{\lambda} \int_{\R} |\psi(y)-\psi(x)|k(x,\dd y),$$
and then 
\begin{align*}
\int_{\R} |L^\mu \psi(x)|^{\alpha}P(\dd x)\le & 2^{\alpha-1} \|V\|_\infty^\alpha\int_{\R}   \max_{|y|\le \|V\|_\infty}|\psi'(x+y)|^\alpha P(\dd x)  \\&+  2^{\alpha-1}\bar{\lambda}^\alpha \int_{\R}\int_{\R} |\psi(z)-\psi(x)|^\alpha k(x,\dd z) P(\dd x) 
\end{align*}
For the first term, we use that $|\psi'(z)|^\alpha=\rho^\alpha |T(z)-z|^{\alpha(\rho-1)}\le \rho^\alpha 2^{(\alpha(\rho-1)-1)^+} (|T(z)|^{\alpha(\rho-1)}+|z|^{\alpha(\rho-1)}).$ Since $T$ is nondecreasing, we get
\begin{align*}
\max_{|y|\le \|V\|_\infty}|\psi'(x+y)|^\alpha \le & \rho^\alpha 2^{(\alpha(\rho-1)-1)^+} \Bigg( \max\left(\left|T\left(x+ \|V\|_\infty \right)\right|^{\alpha(\rho-1)},\left|T\left(x-\|V\|_\infty\right)\right|^{\alpha(\rho-1)}\right) \\
&\phantom{\rho^\alpha 2^{(\alpha(\rho-1)-1)^+}()} +2^{(\alpha(\rho-1)-1)^+} ( |x|^{\alpha(\rho-1)}+\|V\|_\infty^{\alpha(\rho-1)} )\Bigg).
\end{align*}
For the second term, since $\pi(\dd x,\dd y)=P(\dd x)\delta_{T(x)}(\dd y)$ achieves the optimal transport between $P$ and $\tilde{P}$, we get by using~\eqref{ineg_sur_psi} $P(\dd x)P(\dd z)$-a.s., $F$ increasing and the continuity of $\psi$ and $T$ that $|\psi(z)-\psi(x)|\le \max(|T(z)-x|^\rho , |z-T(x)|^\rho)$ for all $x,z\in \R$. We obtain  $|\psi(z)-\psi(x)|^\alpha\le 2^{\alpha \rho -1}(|x|^{\alpha \rho}+|T(x)|^{\alpha \rho}+|z|^{\alpha \rho}+|T(z)|^{\alpha \rho})$ and thus
\begin{align*}
\int_{\R}\int_{\R} &|\psi(z)-\psi(x)|^\alpha k(x,\dd z) P(\dd x) \le 2^{\alpha \rho -1} \Bigg( \int_{\R} |x|^{\alpha \rho} P(\dd x)+  \int_{\R} |x|^{\alpha \rho} \tilde{P}(\dd x) \\
&+ 2^{\alpha \rho -1} \left(M^{\alpha \rho} + \int_{\R} |x|^{\alpha \rho} P( \dd x ) \right)+ \int_{\R}(|T(x+M)|^{\alpha \rho}+ |T(x-M)|^{\alpha \rho})P(\dd x) \Bigg).
\end{align*}
Gathering all the terms and using Lemma~\ref{lem_moments_transport}, we get~\eqref{majo_Lmu}.
\end{proof}

\begin{lemma}\label{lem_convunif}Let $\rho>1$, $r>0$. We consider  $\psi_r(x)=\int_0^x \rho|\tilde{F}_r^{-1}(F_r(y))-y|^{\rho-2}(\tilde{F}_r^{-1}(F_r(y))-y)\dd y$ and $\psi_r^\mu(x)=\int_0^x \rho|(\tilde{F}_r^\mu)^{-1}(F_r^\mu(y))-y|^{\rho-2} [(\tilde{F}_r^\mu)^{-1}(F_r^\mu(y))-y]\dd y$.
Then, under the assumptions of Theorem~\ref{th: main PDMP 1D}, $L^\mu \psi^\mu_r$ converges locally uniformly to $L  \psi_r$. 
\end{lemma}
\begin{proof}
According to Proposition~\ref{prop_outil}, the functions $F_r$, $\tilde{F}_r$, $F^\mu_r$, $\tilde{F}^\mu_r$ are continuous and increasing. The convergence in law stated by Proposition~\ref{prop: convskoro} gives the pointwise convergence of $F^\mu_r$ (resp. $(\tilde{F}^\mu_r)^{-1}$) to $F_r$ (resp. $\tilde{F}_r^{-1}$). Since $x\mapsto (\tilde{F}^\mu_r)^{-1}(F_r^\mu(x))$ is increasing, a classical result sometimes named as the second Dini theorem gives the local uniform convergence of $(\tilde{F}^\mu_r)^{-1}(F_r^\mu(x))$ to $\tilde{F}_r^{-1}(F_r^\mu(x))$. Then, using that $z\mapsto \rho|z|^{\rho-2} z$ is $(\rho-1)$-H\"olderian for $1<\rho\le 2$ and locally Lipschitz for $\rho>2$, we get local uniform convergence of $(\psi^\mu_r)'(x)$ and $\psi^\mu_r(x)-\psi^\mu_r(0)$ toward $\psi'_r(x)$ and $\psi_r(x)-\psi_r(0)$.
We write
\begin{align}
L^\mu \psi^\mu_r(x)-L \psi_r(x)=&\mu\int_0^{\frac{1}{\mu}} [(\psi^\mu_r)'(\Phi(x,s)) -\psi_r'(\Phi(x,s)) ]V(\Phi(x,s))\dd s  \nonumber \\
&+\mu\int_0^{\frac{1}{\mu}}\psi_r'(\Phi(x,s))V(\Phi(x,s))- \psi_r'(x)V(x)\dd s \label{diff_LmuL}\\
&+\lambda(x)\int_{\R}(\psi^\mu_r(y)-\psi_r(y)-\psi^\mu_r(x)+\psi_r(x) )  k(x,\dd y) \nonumber
\end{align}
 Let $a>0$. For $x\in[-a,a]$ and $\mu\ge 1$ and $s\in[0,1/\mu]$, $|\Phi(x,s)|\le a+\|V\|_\infty$. The uniform convergence of  $(\psi^\mu_r)'$ on $[-( a+\|V\|_\infty ), a+\|V\|_\infty]$ and Assumption~\ref{ass:Vlambda}(i) gives the uniform in $x\in [-a,a]$ convergence  to zero  of the first  term of the right-hand-side of~\eqref{diff_LmuL}. Similarly, the third term uniformly converges to zero on  $[-a,a]$ by using the local uniform convergence of $\psi^\mu_r(x)-\psi^\mu_r(0)$, Assumptions~\ref{ass:Vlambda}~(ii) and~\ref{ass:jumps}~(iii). Last, the second term also converges uniformly  to zero on $[-a,a]$ by using that $\Phi$ is continuous on $[-a,a]\times [0,1]$, the Heine-Cantor theorem and the continuity of $\psi'_rV$.
\end{proof}

 \begin{proof}[Proof of Theorem~\ref{th: main PDMP 1D}]
We use the approximating family of pure jump processes $\{X_t^\mu\}_{\mu\geq 1}$ and $\{\oX_t^\mu\}_{\mu\geq 1}$. By using the uniform bound~\eqref{bound_moments} on moments of order $\rho(1+\varepsilon)$ and the convergence in law given by Proposition~\ref{prop: convskoro}, we have by Theorem~6.9 of~\cite{cf:Villani}
$$W_\rho^{\rho}(P_t,\oP_t)-W_\rho^{\rho}(P_0,\oP_0) =\lim_{\mu\to\infty} W_\rho^{\rho}(P^\mu_t,\oP^\mu_t)-W_\rho^{\rho}(P_0,\oP_0).$$
Theorem~\ref{teo:principale} yields 
$$W_\rho^{\rho}(P^\mu_t,\oP^\mu_t)-W_\rho^{\rho}(P_0,\oP_0)=-\int_0^t\left(\int_\R L^\mu\,\psi^\mu_r(x)\, P^\mu_r(\dd x) +\int_\R \oL^\mu\,\opsi^\mu_r(x)\, \oP^\mu_r(\dd x) \right)\dd r. $$
To obtain~\eqref{eq: deri wass PDMPPJ}, it is sufficient show that $\int_\R L^\mu\,\psi^\mu_r(x)\, P^\mu_r(\dd x)$ converges for all $r\in[0,t]$ to $\int_\R L \,\psi_r(x)\, P_r(\dd x)$. Indeed, we can then apply the dominated convergence theorem by using~\eqref{majo_Lmu} and~\eqref{bound_moments}. For $N>0$, we define $[z]_N=\max(-N,\min(z,N))$ for any $z\in \R$, and write

\begin{align*}
  \left|   \int_\R L^\mu\psi^\mu_r(x)\, P_r^\mu(\dd x) -\int_\R L\psi_r \,P_r(\dd x)\right| &\leq 
\left|   \int_\R\left( L^\mu\psi^\mu_r(x)- [L^\mu\psi^\mu_r(x)]_N \right)\, P_r^\mu(\dd x)\right|\\
&+\left|   \int_\R\left(  [L^\mu\psi^\mu_r(x)]_N - [L \psi_r(x)]_N\right)\, P_r^\mu(\dd x)\right| \\
&+\left|\int_\R [L \psi_r(x)]_N (P_r^\mu(\dd x)-P_r(\dd x))\right|\\
&+\left|\int_\R \left([L \psi_r(x)]_N  -L\psi_r(x)\right)P_r(\dd x)\right|
\end{align*}
Let $\eta>0$. Using  $\int_\R |L^\mu\psi^\mu_r(x)|\ind_{\{|L^\mu\psi^\mu_r(x)|>N\}}P_r^\mu(\dd x)\le N^{-\varepsilon} \int_\R |L^\mu\psi^\mu_r(x)|^{(1+\varepsilon)}P_r^\mu(\dd x)$ and \eqref{majo_Lmu} with $\alpha=1+\varepsilon$, we can find $N$ large enough so that the first and the fourth terms are bounded by $\eta/4$, uniformly in $\mu\ge 1$. For the second term, we use that that the family $P_t^\mu$ is tight since it converges weakly to $P_t$. Let $\mathcal{K}$ be a compact subset of $\R$ such that $\int_{\R\setminus \mathcal{K}}P_r^\mu(\dd x) \le \frac{\eta}{16N}$ for any $\mu\ge 1$.  From Lemma~\ref{lem_convunif}, there exists $\mu_\eta$ such that $\sup_{x\in \mathcal{K}}|L^\mu \psi^\mu(x)-L \psi(x)| \le \eta/8 $ for $\mu\ge \mu_\eta$. We get
$$\left|   \int_\R\left(  [L^\mu\psi^\mu_r(x)]_N - [L \psi_r(x)]_N\right)\, P_r^\mu(\dd x)\right|\le 2N \frac{\eta}{16N} +\frac{\eta}{8}=\frac{\eta}{4}.$$
Last, we note that $x\mapsto [L \psi_r(x)]_N$ is bounded and continuous by Assumption~\ref{ass:Vlambda}, Assumption~\ref{ass:jumps}(v) and the continuity of $\psi_r'$. The weak convergence of $P_t^\mu$ gives the existence of $\mu'_\eta\ge\mu_\eta $ such that $\left|\int_\R [L \psi_r(x)]_N (P_r^\mu(\dd x)-P_r(\dd x))\right|\le \eta/4$ for $\mu \ge \mu'_\eta$. We get $\left|   \int_\R L^\mu\psi^\mu_r(x)\, P_r^\mu(\dd x) -\int_\R L\psi_r \,P_r(\dd x)\right|\le \eta$ for $\mu \ge \mu'_\eta$, which yields the claim.
 \end{proof}

 \appendix

 \section{Weak convergence of $X^\mu_t$ to $X_t$}\label{app: skoro}
\begin{proposition}\label{prop: convskoro}
Let Assumptions~\ref{ass:Vlambda} and~\ref{ass:jumps}(iii) and (v) hold.   Then, for any $t\ge 0$, $P_t^{\mu}$ converges weakly to $P_t$ when $\mu \rightarrow +\infty$.
\end{proposition}
\begin{proof}
We start by proving that
$ \{X_t^{\mu}\}_{t \ge 0}$ converges in law to $\{X_t\}_{t \ge 0}$
 as $\mu \rightarrow +\infty$.  We know that for every $0<s<t $, $n \in \N^*$, $s_1,\cdots, s_n\in (0,s)$, 
and all $C^\infty$ functions $g:\R^n\rightarrow \R$ and $f:\R^n\rightarrow \R$ with compact support, we have
\begin{equation}\label{eq: problema martingale}
 \E\left[ \left( f(X_t^\mu)-f(X_s^\mu)-\int_s^t L^\mu f(X_r^\mu)\, \dd r\right) 
 g(X_{s_1}^\mu,\cdots ,X_{s_n}^\mu)\right]=0\,,
\end{equation}
where $L^\mu$ is defined in \eqref{eq: Lmu}.

Let us first write
\begin{align*}  
L^\mu f(x)-Lf   (x) &= \mu \int_0^\frac 1\mu \left[V\left(\Phi\left(x,s\right)\right) 
f'\left(\Phi\left(x,s\right)\right) -V(x)f'(x) \right]\dd s.
\end{align*}
Since $f$ has a compact support, $Vf'$ is bounded and uniformly continuous on this compact by Assumption~\ref{ass:Vlambda}(i). This gives that $L^\mu f$ converges uniformly to $Lf$.

Now, we want to check the tightness criterion of Aldous (see e.g. Theorem~4.5, p.~320 of~\cite{cf:JS}). To do so, we use the representation~\eqref{repr_Xmu}. Let $\capT>0$.  We obtain for $t\in [0,\capT ]$, $|X^\mu_t|\le |X_0|+  \frac{ \|V\|_\infty}{\mu}N^\mu_\capT  +M N_\capT $ and thus the tightness of $(\sup_{t\in [0,\capT ]}|X^\mu_t|)_{\mu\ge 1}$ since $N^\mu_\capT /\mu$ converges in probability to $\capT $ as $\mu \rightarrow + \infty$. Besides, if $S$ and $S'$ denote two stopping times such that $S\le S'\le \min (S+\delta,\capT )$, we have 
$|X^\mu_{S'}-X^\mu_{S}|\le  \|V\|_\infty  ( N^\mu_{S+\delta}- N^\mu_{S})/\mu+M( N_{S+\delta}- N_{S})$. We get by the strong Markov property of Poisson processes
$$\P(|X^\mu_{S'}-X^\mu_{S}|>\eta)\le \P\left(   \|V\|_\infty   N^\mu_{\delta} /\mu >\eta/ 2\right)+ \P\left(M N_{\delta} >\eta/2 \right),$$
for any $\eta>0$.  Let $\epsilon>0$. We choose 
$\delta>0$ small enough such that $\|V\|_\infty \delta<\eta/2$ and $(1-e^{-\bar{\lambda} \delta})<\epsilon/2$.
Since $N^\mu_{\delta}/\mu$ converges in probability to $\delta$, 
 there is $\bar{\mu}$ such that for $\mu\ge \bar{\mu}$,  $ \P\left(   \|V\|_\infty   N^\mu_{\delta} /\mu >\eta/ 2\right) <\epsilon/2$ and thus $\P(|X^\mu_{S'}-X^\mu_{S}|>\eta) < \epsilon.$

 We note $\mathbb{D} $ the space of càdlàg real functions on $[0,+\infty)$ endowed with the Skorokhod topology  and $P^\mu(\dd \omega)$ the probability law of $\{X_t^\mu\}_{t \ge 0}$. We rewrite~\eqref{eq: problema martingale} as
$$\int_{\mathbb{D} } \left( f(\omega(t))-f(\omega(s))-\int_s^t L^\mu f(\omega(r))\, \dd r\right) 
 g( \omega(s_1),\cdots ,\omega(s_n)) P^\mu(\dd \omega)=0.$$
The family $(P^\mu(\dd \omega))_{\mu\ge 1}$ is tight, and we denote by $P^\infty(\dd \omega)$ a limit point of this family. For $r>0$, $\omega\in \mathbb{D}  \mapsto \omega(r)$ is not continuous, but $\mathcal{T}_{P^\infty}=  \{r >0 , P^\infty(\omega(r) \not = \omega(r-))>0 \} $  is a countable set such that for $r \in [0,+\infty) \setminus\mathcal{T}_{P^\infty}$,  $\omega\mapsto \omega(r)$ is $P^\infty(\dd \omega)$-a.s. continuous. Since $f$, $Lf$ and $g$ are bounded continuous by using Assumption~\ref{ass:Vlambda} and~\ref{ass:jumps}(v), we get by using the uniform convergence of $L^\mu f$ to $Lf$ and the weak convergence of a subsequence $P^{\mu_k}$ towards $P^\infty$ that
\begin{equation}\label{eq:mgpb_limite}
\int_{\mathbb{D} } \left( f(\omega(t))-f(\omega(s))-\int_s^t L f(\omega(r))\, \dd r\right) 
 g( \omega(s_1),\cdots ,\omega(s_n)) P^\infty(\dd \omega)=0,
\end{equation}
when $s_1,\dots,s_n,s,t$ are not in $\mathcal{T}_{P^\infty}$. By using the right-continuity of $\omega$, the dominated convergence theorem gives that \eqref{eq:mgpb_limite} still holds for any $0<s<t $, $n \in \N^*$, $s_1,\cdots, s_n\in (0,s)$. Thus, any limit point of $(P^\mu(\dd \omega))_{\mu\ge 1}$ is a  solution of the martingale problem associated to~ $\{X_t\}_{t\ge 0}$. By using Theorem~II$_{13}$ of~\cite{cf:LM}, we have the uniqueness  for the martingale problem, which gives the weak convergence of $\{X_t^{\mu}\}_{t \ge 0}$ to $\{X_t\}_{t\ge 0}$. 

Now, from the representation~\eqref{repr_X}, we easily get 
$\P(X_t=X_{t-})=1$ for any $t>0$. Thus, for all $t\ge 0$, $\omega\in\mathbb{D}  \mapsto \omega(t)$ is $P^\infty$-a.s. continuous, and for any bounded continuous function $h:\R\rightarrow \R$, we have
$\int_{\mathbb{D} }h(\omega(t))  P^\mu(\dd \omega) \underset{\mu \rightarrow + \infty}{\rightarrow} \int_{\mathbb{D} }h(\omega(t))  P^\infty(\dd \omega) $. 
\end{proof}

\section*{Acknowledgements}
This research benefited
     from the support of the ``Chaire Risques Financiers'', Fondation du
     Risque, the French National Research Agency under the program
  ANR-12-BS01-0019 (STAB)


 \end{document}